\documentclass[12pt]{amsart}
\usepackage{amssymb}
\usepackage[all]{xy}
\usepackage{tabulary,caption,comment}
\usepackage{xcolor}
\theoremstyle{plain}

\newtheorem{Thm}{Theorem}[section]
\newtheorem{Cor}[Thm]{Corollary}
\newtheorem{Lem}[Thm]{Lemma}
\newtheorem{Prop}[Thm]{Proposition}

\theoremstyle{definition}
\newtheorem{Def}[Thm]{Definition}

\theoremstyle{remark}
\newtheorem{Rem}[Thm]{Remark}

\newcommand{\C}{\mathbb{C}}
\newcommand{\R}{\mathbb{R}}
\newcommand{\N}{\mathbb{N}}
\newcommand{\B}{\mathbb{B}}
\newcommand{\Aut}{\mathrm{Aut}}
\newcommand{\SAut}{\mathrm{SAut}}
\newcommand{\End}{\mathrm{End}}
\newcommand{\U}{\mathrm{U}}
\renewcommand{\O}{\mathcal{O}}
\makeatletter
\@namedef{subjclassname@2020}{%
  \textup{2020} Mathematics Subject Classification}
\makeatother

\numberwithin{equation}{subsection}

\begin{document}

\title[Holomorphic  Factorization  of Vector Bundle Automorphisms]%
{Holomorphic Factorization of Vector Bundle Automorphisms }
\author{George Ioni\c t\u a \and Frank Kutzschebauch}

\address{Departement Mathematik\\
Universit\"at Bern\\
Sidlerstrasse 5, CH--3012 Bern, Switzerland}

\email{george.ionita@math.unibe.ch}
\email{frank.kutzschebauch@math.unibe.ch}

\thanks{The research was partially supported by Schweizerische Nationalfonds Grant 200021-178730.}

\date{\today}
\setcounter{tocdepth}{3}
\begin{abstract}
We prove that any null-homotopic special holomorphic vector bundle automorphism of a holomorphic rank $2$ vector bundle $E$ over a Stein space $X$ can be written as a finite product of unipotent holomorphic vector bundle automorphisms all well as a finite product of exponentials.
\end{abstract}
\maketitle
\bibliographystyle{amsalpha}
\tableofcontents
\section{Introduction}
\label{introduction}

We consider a holomorphic vector bundle $E \to X$ over a Stein space $X$.
By $E_x$ we denote the fibre of the bundle $E$ over the point $x \in X$.
Let $F: X \to \SAut(E)$ be a special holomorphic vector bundle automorphism of $E$. Here special refers to the fact that the holomorphic function $\det F: X \to \C^\ast$ is constantly $1$. Recall from \cite{Ca} that $F$ is a global section in the automorphism bundle $\Aut(E)$, which is a holomorphic
fibre bundle with typical fibre $\operatorname{GL}_n(\C)$, and it is a group bundle in the sense of Cartan \cite{Ca}. The condition that the determinant of $F$ (recall that the determinant of an endomorphism of a vector space is defined independently of the choice of a basis of the vector space) can be formulated as $F$ being a section in the subbundle $\SAut(E) \subset \Aut(E)$, which is again a  group bundle, where the fibre $\operatorname{GL}_n(\C)$ has been replaced by $\operatorname{SL}_n(\C)$, i.e., the transition functions (acting by conjugation = outer automorphisms of $\operatorname{GL}_n(\C)$) remain the same. We have the corresponding bundles of Lie algebras $\End^0(E) \subset \End(E)$, where the superscript $0$ refers to the trace zero condition $\mathfrak{sl}_n(\C) : = \{A \in \mathfrak {gl}_n(\C) : \operatorname{trace}A = 0\}$. There is a global bundle map (i.e., fibrewise map), the exponential map $\exp : \End(E) \to \Aut(E)$ which restricts to the subbundles: $\exp : \End^0(E) \to \SAut(E)$. 

We call an automorphism of $E$, $\alpha \in \Aut(E)$, \textit{unipotent} if and only if $\alpha - \operatorname{Id}$ is nilpotent, viewed as an endomorphism of the fibres of $E$, i.e., for each point $x$, the linear map $\alpha (x) - \operatorname{Id}_{E_x} : E_x \to E_x$ is nilpotent. Clearly such $\alpha$ is necessarily of determinant $1$, i.e., $\alpha \in \SAut(E)$. We denote the subset of unipotent automorphisms by $\U(E)$. 

Our main result is the following:
\begin{Thm}\label{Thm:main} Let $X$ be a Stein space and $E \to X$ a rank $2$ holomorphic vector bundle over $X$. Then $F \in \SAut(E)$ is a (finite) product of unipotent holomorphic automorphisms $u_i \in \U (E)$, $i = 1, 2, \ldots, K$,
$$F(x) = u_1 (x) \cdot u_2 (x) \cdot \ldots \cdot u_K (x)$$
if and only if $F$ is null-homotopic.
\end{Thm}

In the case of a trivial vector bundle $E \cong X \times \C^n$, the special holomorphic automorphisms correspond to holomorphic maps $f : X \to \operatorname{SL}_n (\C)$. There exist "canonical" unipotent automorphisms, namely the maps into upper triangular unipotent matrices (i.e., with $1$'s on the diagonal) or into lower diagonal unipotent matrices. 

The factorization of matrices with entries in a ring $R$ into intertwining products of upper and lower triangular unipotent matrices with entries in the ring $R$ is a widely studied subject. In the present paper, $R = \O(X)$ is the ring of holomorphic functions on the Stein space $X$. 
For rings of polynomial functions the  most famous result is due to Suslin \cite{Suslin}, who proves that every matrix in $\operatorname{SL}_n (\C [z_1, z_2, \ldots z_k])$ is a product of upper/lower triangular matrices with entries in $\C[z_1, z_2, \ldots, z_k]$ if $n \ge 3$. 
Cohn in \cite{Cohn} had already proven that this is false for $\operatorname{SL}_2 (\C[z_1, z_2, \ldots z_k])$ and $k \ge 2$.
For rings of continuous functions, the problem was solved by Thurston and Vaserstein \cite{Thurston}. 
For the ring $\O(\C^n)$ of holomorphic functions on $\C^n$, the factorization problem was posed by Gromov in \cite{Gromov:1989}. He called it the "Holomorphic Vaserstein Problem" and it was solved for general Stein manifolds in \cite{Ivarsson:2012}. 
Since any vector bundle over an open Riemann surface is trivial, our main theorem for $X$ a Riemann surface is already contained in the main result from that paper, moreover it had already been proven in 1986 in \cite{Klein:1988}.

The problem is known in linear algebra under the name of LU-factorization.
It is  studied not only over fields, but also very extensively  over rings \cite{VS} and the question of existence of such a factorization is formalized in the definition of  the first (special) $K1$-group of a ring (see \cite{Bass}). 
For Banach algebras the factorization result is proved in the the textbook of Milnor \cite{Milnor:1971}. Factorization over rings is connected to the notion of Bass stable rank, a sufficient condition for a ring $R$ to solve the factorization for $\operatorname{SL}_n(R)$ for all $n \ge 2$ is that the Bass stable rank of the ring is equal to $1$. Examples of such rings in the context of complex analysis are for instance the disc algebra \cite{Jones:86} or more general the continuous functions on a bordered Riemann surface which are holomorphic in the interior \cite{Leiterer}.

For the above mentioned cases of rings of polynomial, continuous or holomorphic functions, one can view the factorization as a parametrized version of the Gauss elimination process and a necessary and sufficient condition for its solution is the null-homotopy of the map into $\operatorname{SL}_n(\C)$. 

The generalization from products of upper/lower triangular to products of unipotent matrices was used in \cite{KuSt} for proving that any matrix in $\operatorname{SL}_2(\O(\C))$ is a product of two exponentials. 

All results described above can be formulated in $K$-theoretic terms for matrices over rings. 
The situation of vector bundles considered in this paper goes beyond the framework of $K$-theory. 
The first generalization to the present situation of vector bundle automorphisms was done in the topological case by Hultgren and Wold \cite{WoHu}. They adapted Vaserstein's strategy from \cite{Vaserstein} to the setting of vector bundles. Unfortunately, we cannot rely on their  result since they do not prove that a product of "continuous replica"
of a fixed beforehand constructed set of unipotent holomorphic  automorphisms does the job. On the other side, our strategy from Subsections \ref{ssection:topologicalsuitablepairs} and \ref{ssection:topologicalULfactorization} can be adapted to the case of continuous vector bundles of any rank over finite dimensional topological manifolds to prove the result of Hultgren and Wold. We leave the workout of the details to the interested reader.  

The strategy of proof is as follows. Since our vector bundle $E$ is not necessarily trivial, the global unipotent automorphisms given by upper/lower triangular unipotent matrices in the case of plane maps $X \to \operatorname{SL}_n(\C)$ do not exist. We have to replace them in the holomorphic category. This construction is done in Section \ref{section:pairs}. If the rank $2$ vector bundle $E$ does not admit a trivial subbundle, then holomorphic unipotent automorphisms of $E$ have to "degenerate" to the identity $\operatorname{Id}_{E_x}$ for $x$ contained in a (codimension $1$) analytic subset. We construct the holomorphic unipotent automorphisms in pairs of "upper/lower" triangular automorphisms $(U_i^+, U_i^-) \in \U(E) \times \U(E)$, where $U_i^+ = \operatorname{Id} + N_i^+$, $U_i^- = \operatorname{Id} + N_i^-$ have  the property that $N_i^+, N_i^- \in \End (E)$ are divisible by a holomorphic function $f_i \in \O(X)$, $i= 1, 2, \ldots, n+1$ such that the $f_i$'s have no common zeros. On the set $X \setminus \{f_i = 0\}$ the bundle $E$ is trivial and the pair $(N_i^+, N_i^-)$ is conjugated to the "standard" pair 
\[\left( \begin{array}{cc}
    0 & 1 \\
    0 & 0
\end{array} \right),\
\left( \begin{array}{cc}
    0 & 0 \\
    1 & 0
\end{array}\right).
\]
We then proceed in two steps.

In the first step in Section \ref{section:suitablepairs}, we factorize $F$ as a product of null-homotopic special automorphisms $G_i$,
\[F = G_1 \cdot G_2 \cdot \ldots \cdot G_{n+1},
\]
such that $G_i - \operatorname{Id}$ is divisible by $f_i^4$, and the null-homotopy satisfies the condition $(G_i)_t (x) = \operatorname{Id}$ when $f(x)=0$. Such an automorphism $G_i$ is called \textit{suitable for the pair} $(U_i^+, U_i^-)$. In this step we use the Oka principle for Oka Pairs due to Forster and Ramspott. However, we choose to present our proofs using a homotopy principle due to Studer. The technical details needed as input in Studer's Theorem, are quoted from the paper of Forster and Ramspott \cite{Okasche Paare}. We hope this helps readers who are not specialists in Oka Pairs to understand the material better and give them a much more general view on Oka principles.

The solution to the corresponding topological problem is done with bare hands in Subsection \ref{ssection:topologicalsuitablepairs}.

In the second step in Section \ref{section:factorizationofsuitableautomorphisms}, we factorize each of the $G_i$'s suitable for the pair $(U_i^+, U_i^-)$ into an intertwining product of holomorphic "replica" of this pair, i.e., there exist $h_i \in \O(X)$ such that
\begin{equation}\label{topsol}\tag{$\square$}
G_i(x) = U_i^- (h_1(x)) \cdot U_i^+ (h_2(x)) \cdot \ldots.
\end{equation}
Here a \textit{replica} $U(f)$ of a unipotent automorphism $U = \operatorname{Id} + N \in \Aut(E)$ (by a function $f$) is simply defined as $U(f(x)) = \operatorname{Id} + f \cdot N$.

This follows the strategy from \cite{IK2} and it is limited to rank $2$ vector bundles. It is the only place in the paper where we heavily use the  rank $2$ condition on the vector bundle. All other results are proven (or can be proven without a big extra effort) for vector bundles of any rank $n \ge 2$. The topological solution for the factorization (\ref{topsol}) is given in Subsection \ref{ssection:topologicalULfactorization} and relies on a special factorization result for continuous complex valued matrices "near" to the identity matrix, namely Lemma \ref{lem:novelties}. In particular, we
give a new form of the Whitehead Lemma for rings of complex valued continuous functions, Lemma \ref{lem:novelties} (3). 

The holomorphic result is then deduced from the topological one  using the Oka principle for stratified elliptic submersions in Subsection \ref{ssection: holomorphicULfactorization}. 

 Finally, let us mention an immediate corollary (observed by Doubtsov and the second author in \cite{DK}) of our main result to the so called  product-of-exponentials decomposition introduced by Mortini and Rupp in \cite{Mort}. For that denote $\exp : \operatorname{End}^0(E) \to \SAut(E)$
the exponential map from trace zero endomorphisms (the Lie algebra bundle) to  the bundle of special automorphisms. Realizing that for a nilpotent automorphism $N$, the Taylor formula of $\log(U) \log (\operatorname{Id}+N)$ is just a polynomial in $N$, thus converges globally on $X$, we have proven the following.

\begin{Cor} Let $X$ be a Stein space and $E \to X$ a rank $2$ holomorphic vector bundle over $X$. Then $F \in \SAut (E)$ is a (finite) product of  holomorphic exponentials, i.e., there are  $a_i \in \Gamma (\End^0(E), X) $, $i= 1,2, \ldots, K$,
\[F(x) = \exp (a_1) (x) \cdot \exp (a_2) (x) \cdot \ldots \cdot \exp (a_K) (x)
\]
if and only if $F$ is null-homotopic.
\end{Cor}

\section{Preliminaries}
\label{section:preliminaries}

The main technical tool we use are two instances of the Oka principle. One  
of them, the Oka principle for sections of stratified elliptic submersions is well known and prominent in the modern literature. To present it we need to introduce the concept of a spray associated with a holomorphic submersion following \cite{Gromov:1989} and \cite{Forstneric:2002}. We start with some notation and terminology. Let $h\colon Z \to X$ be a holomorphic submersion of a complex manifold $Z$ onto a complex manifold $X$. For any $x\in X$, the fibre over $x$ of this submersion will be denoted by $Z_x$. At each point $z\in Z$, the tangent space $T_zZ$ contains {\it the vertical tangent space} $VT_zZ=\ker Dh$. For holomorphic vector bundles $p\colon E \to Z$, we denote the zero element in the fibre $E_z$ by $0_z$.

\begin{Def}
\label{definspray}

Let $h\colon Z \to X$ be a holomorphic submersion of a complex manifold $Z$ onto a complex manifold $X$. A \textit{spray} on $Z$ associated with $h$ is a triple $(E,p,s)$, where $p\colon E\to Z$ is a holomorphic vector bundle and $s\colon E\to Z$ is a holomorphic map such that for each $z\in Z$ we have 
\begin{itemize}
\item[(i)]{$s(E_z)\subset Z_{h(z)}$,}
\item[(ii)]{$s(0_z)=z$, and}
\item[(iii)]{the derivative $Ds(0_z)\colon T_{0_z}E\to T_zZ$ maps the subspace $E_z\subset T_{0_z}E$ surjectively onto the vertical tangent space $VT_zZ$.}
\end{itemize} 

\end{Def} 

\begin{Rem}
\label{fibredominating}

We will also say that the submersion \textit{admits a spray}. A spray associated with a holomorphic submersion is sometimes called a \textit{(fibre) dominating spray}.

\end{Rem}

One way of constructing dominating sprays, as pointed out by Gromov, is to find finitely many $\mathbb{C}$-complete vector fields that are tangent to the fibres and span the tangent space of the fibres at all points in $Z$. One can then use the flows $\varphi_j^t$ of these vector fields $V_j$ to define $s \colon Z \times \mathbb{C}^N \to Z$ via $s(z, t_1, \ldots, t_N) = \varphi_1^{t_1} \circ \cdots \circ \varphi_N^{t_N}(z)$ which gives a spray. 

\begin{Def}

Let $X$ and $Z$ be complex spaces. A holomorphic map $h \colon Z \to X$ is said to be a \textit{submersion} if for each point $z_0\in Z$ it is locally equivalent via a fibre preserving biholomorphic map to a projection $p\colon U\times V \to U$, where $U\subset X$ is an open set containing $h(z_0)$ and $V$ is an open set in some $\mathbb{C}^d$.  

\end{Def}

We will need to use stratified sprays which are defined as follows.

\begin{Def} 
\label{definstratifiedspray}

We say that a submersion $h \colon Z \to X$ \textit{admits stratified sprays} if there is a descending chain of closed complex subspaces $X = X_m\supset \cdots \supset X_0$ such that each stratum $Y_k = X_k \setminus X_{k-1}$ is regular and the restricted submersion $h \colon Z|_{Y_k} \to Y_k$ admits a spray over a small neighborhood of any point $x\in Y_k$.  

\end{Def}

\begin{Rem}

We say that the stratification $X = X_m\supset \cdots \supset X_0$ is \textit{associated with the stratified spray}.

\end{Rem}

In \cite{Forstneric:2001}, see also \cite[Theorem 8.3]{Forstneric:2010}, the following theorem is proved. 
\begin{Thm}
\label{t:forstneric}

Let $X$ be a Stein space with a descending chain of closed complex subspaces $X = X_m\supset \cdots \supset X_0$ such that each stratum $Y_k = X_k \setminus X_{k-1}$ is regular. Assume that $h \colon Z \to X$ is a holomorphic submersion which admits stratified sprays then any continuous section $f_0\colon X\to Z$ such that $f_0|_{X_0}$ is holomorphic can be deformed to a holomorphic section $f_1\colon X \to Z$ by a homotopy that is fixed on $X_0$. 

\end{Thm}

 The second Oka principle that we (mainly) use is the Oka principle for Oka pairs due to Forster and Ramspott. However, it is not  prominent in the modern literature on Oka theory and the original paper \cite{Okasche Paare} is in German. We therefore choose to deduct it from a general unifying approach to proofs for homotopy (and thus Oka) principles presented by Studer \cite{Studer2}.
 
To formulate his result we need to recall some basic notions of complex geometry. Let $X$ be a reduced complex space. A compact $C \subset X$ is called a \textit{Stein compact} if it admits a basis of Stein neighborhoods in $X$. If $C \subset B$ for some $B \subset X$, then the compact $C$ is called $\O (B)$-\textit{convex} if $ C = \{p \in B : \vert f (p) \vert \le \displaystyle\max_{ x\in C} \vert f(x) \vert, f \in \O(B)\}$. Here $\O(B)$ denotes the set of all holomorphic functions defined on unspecified neighborhoods of $B$.

\begin{Def}

Let $X$ be a complex space and let $A, B \subset  X$. The ordered pair $(A, B)$ is a $C$-\textit{pair} if

\begin{enumerate}
    \item  $A, B, A \cap B$ and $A \cup B$ are Stein compacts;
    \item  $A \cap B$ is $\O(B$)-convex;
    \item $A \setminus B \cap B \setminus A = \emptyset$.
\end{enumerate}
 
\end{Def}

The notion of a sheaf of topological spaces used by Studer is the following.

\begin{Def}

Let $X$ be a topological space. A \textit{sheaf of topological spaces} on $X$ is a sheaf $U \mapsto \Phi(U)$, $U \subset X$ open, whose sets of local sections are topological spaces, whose restrictions are continuous, and which is well-behaved in the sense that for the closed unit ball $\B \subset \R^n$ of any real dimension $n \ge 1$ the presheaf $U \mapsto \Phi^\B(U)$ is in fact a sheaf. Here $\Phi^\B(U)$ denotes the set of continuous maps $\B \mapsto \Phi(U)$.

\end{Def}

A sheaf of topological spaces is said to be \textit{metric} if every set of local sections is a metric space, and a metric sheaf is said to be \textit{complete} if every set of local sections is a complete metric space. We say that $\Phi \hookrightarrow \Psi$ is an \textit{inclusion} of sheaves of topological spaces if $\Phi$ is a subsheaf of the sheaf of topological spaces $\Psi$ whose sets of local sections are endowed with the subspace topology of the sets of local sections of $\Psi$.

The following properties for sheaves of topological spaces are important for Oka principles. In the following $s$ and $t$ denote real numbers in the unit interval and by using them as subscripts, we always indicate a homotopy in the common sense. We use the convention that a singleton (with empty boundary) is the closed unit ball of dimension $n = 0$. For a subset $A$ of $X$ we denote by $\mathring A$ its interior. A restriction $\Phi(U) \to \Phi(V)$ of a sheaf $\Phi$ will be denoted by $r_V$ if there is no ambiguity regarding the domain.

\begin{Def}

An inclusion $\Phi \hookrightarrow \Psi$ of sheaves of topological spaces over a topological space $X$ is a \textit{local weak homotopy equivalence} if for every point $p \in X$, every open neighborhood $U$ of $p$ and every continuous map $f$ from the closed unit ball $\B \subset  \R^n$ of any real dimension $n \ge 0$ to $ \Psi(U)$ with $f\vert_{\partial \B} \subset \Phi (U)$, there is a neighborhood $p \in  V \subset U$ and a homotopy $f_t : \B \to  \Phi(V)$ such that $f_0 = r_V   \circ  f $, $f_t \vert_{\partial \B}$ is independent of $t$ and $f_1$ has values in $\Phi(V)$.

\end{Def}

\begin{Def}
\label{Def:weaklyflexible} 

Let $\Phi$ be a sheaf of topological spaces on $X$ and let $A, B \subset X$ be compact sets. The ordered pair $(A, B)$ is called \textit{weakly flexible} for $\Phi$ if the following holds. Given open neighborhoods $U$, $V$ and $W$ of $A$, $B$ and $A \cap B$, respectively and a triple of maps $a : \B \to \Phi (U)$, $b : \B \to \Phi (V)$ and $c_s : \B \to \Phi (W)$ such that $c_0 = r_W \circ a$, $c_1 = r_W \circ b$ and $c_s\vert_{\partial \B}$ is independent of $s$, there are smaller neighborhoods $A \subset U^\prime \subset U$, $B \subset V^\prime \subset V$ and $A \cap B \subset  W^\prime \subset W$, respectively and homotopies $a_t : \B \to \Phi (U^\prime)$, $b_t : \B \to \Phi(V^\prime)$ and $c_{s,t} : \B \to  \Phi(W^\prime)$ with $a_0 = r_{U^\prime} \circ a$, $b_0 = r_{V^\prime} \circ b$ and $c_{s,0} = r_{W^\prime} \circ c_s$, respectively such that

\begin{enumerate}
\item $c_{0,t} = r_{W^\prime} \circ  a_t$ and $c_{1,t} = r_{W^\prime} \circ b_t$,
\item the restrictions $a_t\vert_{\partial B}$, $b_t\vert_{\partial B}$ and $c_{s,t}\vert_{\partial B}$ are independent of $t$,
\item $c_{s,1}$ is independent of $s$,
\item $r_{\mathring A} \circ  a_t$ is in a prescribed neighborhood of $r_{\mathring A} \circ  a_0 : \B \to \Phi (\mathring A)$ with respect
to the compact open topology for all $t$.
\end{enumerate}

If $(A, B)$ is a weakly flexible pair for $\Phi$ such that the homotopy $a_t$ from the conclusion
of the definition can be chosen to satisfy $r_{\mathring A} \circ a_t =
r_{\mathring A} \circ a_0$ for all $t$, then $(A, B)$ is
called an \textit{ordered flexible pair}.

\end{Def}

Here comes the general homotopy theorem due to Studer.

\begin{Thm}
\label{Thm:Studer}

Let $X$ be a second countable reduced Stein space and let $\Phi \hookrightarrow \Psi$ be a local weak homotopy equivalence of sheaves of topological spaces on $X$. Assume that one of the two following statements holds:
\begin{enumerate}
\item 
$\Phi$ is complete metric and every point $p \in X$ has a neighborhood $U$ such that every $C$-pair $(A, B)$ with $B \subset U$ is weakly flexible for $\Phi$.
\item Every point $p \in X$ has a neighborhood $U$ such that every $C$-pair $(A, B)$ with $B \subset U$ is ordered flexible for $\Phi$.
\end{enumerate}
Then $\Psi(X) \ne \emptyset$ if and only if $\Phi(X) \ne \emptyset$. Moreover, if $\Psi$ is likewise in either of the two classes of sheaves, then $\Phi (X) \hookrightarrow \Psi(X)$ is a weak homotopy equivalence.
\end{Thm}

\begin{Rem} 

Since analytic continuation is unique, assumption $(2)$ from Theorem \ref{Thm:Studer} is too strong to be satisfied by a sheaf of analytic maps. Thus it is more suitable for sheaves of continuous maps, where cutoff functions can be used. Assumption $(1)$ is the property that one tries to show if $\Phi$ is a complex analytic sheaf. 

\end{Rem} 

\section{Unipotent Holomorphic Automorphisms. Construction of the Pairs}
\label{section:pairs}

Let $\pi : E \rightarrow X$ be a holomorphic vector bundle of rank $2$, where $X$ is a Stein space of dimension $n$. In this section we will construct global unipotent  holomorphic automorphisms (equivalently nilpotent endomorphism) of $E$. 

We start with some preliminary thoughts. Let $N \in \End_{\text{hol}}(E)$ be a global nilpotent endomorphism. Any nilpotent endomorphism of a $2$-dimensional complex vector space is either the zero map or it is conjugated to 
\[\left( \begin{array}{cc}
    0 & 0 \\
    1 & 0
\end{array}\right).
\]
The corresponding basis of the vector space is called a \textit{Jordan basis}. If there was a global holomorphic Jordan basis for $N$ (i.e., two holomorphic sections $s_1$, $s_2$ so that $s_1(x), s_2(x)$ is a Jordan basis with respect $N_x : E_x \to E_x$ such that the second case occurs in each point $x \in X$), then the vector bundle $E$ would contain a trivial line subbundle defined by $\operatorname{Ker}(N)$. Since not all holomorphic rank $2$ vector bundles over a Stein space contain a trivial holomorphic line bundle, we cannot expect this situation to hold in general. Therefore, in general, any nilpotent endomorphism $N \in \End_{\text{hol}}(E)$ has to "degenerate" to the zero map on some analytic subset of $X$.

In order to construct $N$, we start with two sections $s_1, s_2 \in \Gamma_{\text{hol}}(E, X)$ which are linear independent outside a proper analytic subset of $X$. Since sections of a coherent sheaf can be extended from analytic subsets we can choose any discrete subset $T$ of $X$, in which $s_1(x), s_2(x)$ form a basis of $E_x$. 

Let $E$ have a trivialization $X = \cup_{i=1}^\infty U_i$ with a cocycle of transition functions $f_{i,j} : U_{i,j} \to \operatorname{GL}_2(\C)$. Let $s_1^i, s_2^i : U_i \to \C^2$ be the local representations of the two sections. Let $A$ be the set of points in $X$ such that $s_1(x)$ and $s_2(x)$ are linearly dependent. On $U_i$ this set is described by the vanishing of a determinant. So, $A\cap U_i = \{x \in U_i : \det(s_1^i (x), s_2^i (x)) = 0$\} (showing that $A$ is analytic of codimension $1$ in $X$). 
Since $f_{i,j} s_1^i(x)=  s_1^j(x)$ and $f_{i,j} s_2^i(x)=  s_2^j(x)$,
the local determinants transform by the cocycle 
\[\alpha_{i,j} := \det(f_{i, j}) = \det(s_1^j(x), s_2^j(x)) (\det(s_1^i(x), s_2^i(x)))^{-1} : U_{i,j} \to \C^\ast.
\]
The holomorphic line bundle $L$ defined by this cocycle may not be trivial
and thus there may not be a holomorphic function $f \in \mathcal{O}(X)$ which solves the Cousin problem corresponding to the Cousin distribution $\det (s_1^i(x), s_2^i(x))$. However, any global holomorphic section in $L^{-1}$ gives a holomorphic function $f \in \mathcal{O}(X)$ with the property that on $U_i$, the quotient $f(x) (\det (s_1^i(x), s_2^i(x)))^{-1}$ is holomorphic. Indeed, if on $U_i$, the section is given by the holomorphic map $f_i : U_i \to \C$, then on $U_{i, j}$, we have $f_i (x) \det (s_1^i(x), s_2^i(x)) (\det (s_1^j(x), s_2^j(x)))^{-1} = f_j(x)$, thus if on $U_i$, we define $f(x) = f_i (x) \det(s_1^i(x), s_2^i(x))$, then the function $f$ is globally well defined and, as desired on $U_i$, we get that $f(x) \det(s_1^i(x), s_2^i(x))^{-1} = f_i(x)$ is holomorphic.

Again by the fact that sections of a coherent sheaf can be extended from analytic subsets, we may choose the section of the line bundle such that the function $f$ does not vanish on the (same as above) discrete subset $T$ of $X$.

Now we can define a global nilpotent holomorphic endomorphism $N^-$ by
\[N^-(s_1(x)) = f(x)  s_2(x)\, \text{and}\, N^-(s_2(x)) = 0
\]
in points $x$ where $s_1(x), s_2(x)$ form a basis of $E_x$. In order to
see that this extends holomorphically to all of $X$, let us calculate it
in $U_i$ in points where $s_1(x), s_2(x)$ form a basis of $E_x$, i.e., where the vectors $s_1^i(x), s_2^i(x)$ form a basis of $\C^2$.
If we denote 
\[s_1^i(x) = \left(\begin{array}{c}
a(x) \\ c(x) \end{array}\right)\, \text{and}\, s_2^i(x) = \left(\begin{array}{c} b(x) \\ d(x) \end{array}\right),
\]
then in the standard basis $e_1, e_2 \in \C^2$ the matrix of $N^-$ is given by
\[\left(\begin{array}{cc}
    a(x) & b(x) \\
    c(x) & d(x) 
\end{array}\right)
\cdot 
\left(\begin{array}{cc}
      0 & 0 \\
    f(x) & 0 
\end{array}\right)
\cdot
\left(\begin{array}{cc}
    a(x) & b(x) \\
    c(x) & d(x) 
\end{array}\right)^{-1} = 
\]
\[= \frac{f(x)}{\det(s_1^i(x), s_2^i(x))} \left( \begin{array}{cc}
    a(x) & b(x) \\
    c(x) & d(x) 
\end{array}\right)
\cdot 
\left(\begin{array}{cc}
      0 & 0 \\
     1 & 0 
\end{array}\right)
\cdot
\left(\begin{array}{cc}
    d(x) & -b(x) \\
    -c(x) & a(x) 
\end{array}\right).
\]

By our choice of $f$ we see that this extends holomorphically to all of $U_i$, thus $N^-$ is globally well defined. In points $x$ where $s_1(x), s_2(x)$ form a basis and $f(x)$ vanishes, $N^-$ is the zero map, and in points $x$ where $s_1(x), s_2(x)$ are linearly dependent, it may be zero or have rank $1$. In points where $f$ does not vanish, it has rank $1$.

Defining $N^+$ by
\[N^+(s_1(x)) =  0\, \text{and}\, N^+(s_2(x)) = f(x) s_1(x),
\]
we have proved the following lemma.

\begin{Lem}
\label{Lemma:pair}

Let $\pi : E \rightarrow X$ be a holomorphic vector bundle of rank $2$ over a Stein space $X$. Given a discrete subset $T$ of $X$, there is a pair of nilpotent holomorphic automorphisms $N^+, N^-$ in $\operatorname{End}(E)$ and a (non-zero) holomorphic function $f \in \mathcal{O}(X)$, nowhere vanishing on $T$, with the property that on $X \setminus \{f = 0\} \supset T$ the bundle $E$ is trivial and locally the pair $(N^+, N^-)$ is holomorphically conjugate to the "standard" pair
\[\left(\begin{array}{cc}
    0 & 1 \\
    0 & 0
\end{array} \right),
\left( \begin{array}{cc}
    0 & 0 \\
    1 & 0
\end{array}\right).
\]
\end{Lem}

By standard induction over dimension we can deduce the following.

\begin{Prop}
\label{existence of pairs} 

Let $\pi : E \rightarrow X$ be a holomorphic vector bundle of rank $2$ over a Stein space $X$ of dimension $n$. Then there exist $n+1$ pairs of nilpotent holomorphic automorphism $N_i^+, N_i^-$ in $\End(E)$, $i= 1, 2, \ldots, n+1$ together with (non-zero) holomorphic functions $f_i \in \mathcal{O}(X)$, without common zeros, with the property that on each of the sets $X \setminus \{f_i =0\}$ the bundle $E$ is trivial and locally the pair $(N_i^+, N_i^-)$ is holomorphically conjugated on $X \setminus \{f_i = 0\}$ to the "standard" pair
\[\left(\begin{array}{cc}
    0 & 1 \\
    0 & 0
\end{array}\right),
\left(\begin{array}{cc}
    0 & 0 \\
    1 & 0
\end{array}\right).
\]

\begin{proof}

Choose a discrete set $T$ with the property that it contains at least one point in the smooth part of any irreducible component of $X$. Lemma \ref{Lemma:pair} provides us with a pair of nilpotent holomorphic automorphisms $N_1^+, N_1^-$ in $\End(E)$ and with a function $f_1 \in \mathcal{O}(X)$ such that the bundle $E$ is trivial on $X \setminus \{f_{1} = 0\}$ and locally the pair $(N_1^+, N_1^-)$ is holomorphically conjugate to the "standard" pair. Moreover, by the choice of $T$, the analytic set 
$\{f_{1} = 0\}$ has strictly smaller dimension than $X$, i.e., at most $n-1$. Here by dimension of a complex space we mean the dimension of the smooth part (which equals the maximum of the the dimensions of the smooth parts of the irreducible components).
 
 Now we are going to apply Lemma \ref{Lemma:pair} for a discrete set $T \subset \{f_{1} = 0\}$ which is chosen so that it contains at least one point from the smooth part of each irreducible component of $\{f_{1} = 0\}$ . We get a second pair of nilpotent holomorphic automorphisms $N_2^+, N_2^-$ in $\End(E)$ and a function $f_2 \in \mathcal{O}(X)$ that is not identically vanishing on any irreducible component of $\{f_{1} = 0\}$ and such that the bundle $E$ is trivial on $X \setminus \{f_{2} = 0\}$ and locally the pair $(N_2^+, N_2^-)$ is holomorphically conjugated to the "standard" pair. Observe that the common zero set $\{f_{1} = 0\} \cap \{f_{2} = 0\}$ has dimension strictly smaller than the dimension of $\{f_{1} = 0\}$, i.e., at most $n-2$. After $n$ steps we have that $\dim (\{f_{1} = 0\} \cap \cdots \cap \{f_{n} = 0\})$ is at most zero, thus a discrete set. Using one last time Lemma \ref{Lemma:pair} for $T= \{f_{1} = 0\} \cap \cdots \cap \{f_{n} = 0\}$, provides us with the $n+1$-th pair and concludes the proof.
 
\end{proof}

\end{Prop}

Defining $U_i^+ = \operatorname{Id} + N_i^+, U_i^-=\operatorname{Id} + N_i^-$, we find pairs of unipotent  holomorphic automorphism $U_i^+, U_i^-$ in $\operatorname{U}(E)$, $i = 1, 2, \ldots, n+1$ together with (non-zero) holomorphic functions $f_i \in \mathcal{O}(X)$, without common zeros, i.e., $\{x \in X : f_1(x) = f_2(x) = \cdots = f_{n+1}(x) = 0\} $ with the property that on each of the sets $X \setminus \{f_i =0\}$, the bundle $E$ is trivial and the pair $(U_i^+, U_i^-)$ is holomorphically conjugated to the "standard" pair 
\[\left( \begin{array}{cc}
    1 & 1 \\
    0 & 1
\end{array}\right),
\left( \begin{array}{cc}
    1 & 0 \\
    1 & 1
\end{array}\right).
\]

For a  function $h: X \to \C$ (we will only use continuous or holomorphic functions) and a nilpotent endomorphism $N \in \End(E)$ with $U = \operatorname{Id}+ N$ we define the "replica" $N(h) := h \cdot N$ and $U(h):= \operatorname{Id}+ N(h)$. 

Given $(U_1^+, U_1^-), \ldots, (U_{n+1}^+, U_{n+1}^-)$,
our final goal is to write any null-homotopic special holomorphic automorphism $F$ of $E$ as a finite  product of unipotent automorphisms of the form $U_i^-(z_{2k-1}) U_i^+(z_{2k})$ for holomorphic functions $z_{2k-1}$, $z_{2k}$, $k = 1, 2, \ldots, N_i$ using all the pairs $(U_i^+, U_i^-)$, for a $i = 1, 2, \ldots, n+1$. That is
\[F = \prod_{i=1}^{n+1}\prod_{k=1}^{N_i} U_i^-(z_{2k-1}) U_i^+(z_{2k}).
\]

\section{Factorization into Automorphisms Suitable for a Pair}
\label{section:suitablepairs}

As before, $\pi : E \to X$ is a holomorphic vector bundle over a Stein space $X$. As we have seen in Section \ref{section:pairs}, a pair $(U^+, U^-)$ of unipotent automorphisms of $E$ may "degenerate" to the identity in certain points (e.g., the points where the function $f$ vanishes and the sections $s_1, s_2$ used for the construction of the pair form a basis) contained in a subvariety $\{f = 0\}$ for some $f \in \mathcal{O}(X)$.
 
Thus any product of unipotent automorphisms, being replica of a pair, is necessarily equal to the identity in such points. In this section we will formulate a necessary condition for a special holomorphic automorphism to be written as a product of unipotent automorphisms being replica of a pair. We call this \textit{being suitable for the pair}. Then we prove that any special holomorphic automorphism of $E$ can be written a product of special automorphisms $G_i$, each $G_i$ being suitable for the pair $(U_i^+, U_i^-)$.

Since calculations show that the matrix of such a product of replica of a unipotent pair  differs from the identity by a matrix divisible by $f^3$, we can write a special automorphism $G$ as a product of replica of a pair only when it satisfies this divisibility condition. Here we say that $G \in \End(E)$ is \textit{divisible} by $g$, if every entry of in the matrix $M_G (x)$ is divisible by $g$, where $M_G(x) \in \operatorname{Mat_n}(\C)$ is the matrix representing $G$ with respect to a local trivialization of the vector bundle $E$. Obviously, this condition does not depend on the trivialization. We will write $g \mid (G-\operatorname{Id})$. For our solution to the continuous problem in Subsection \ref{ssection:topologicalULfactorization}, we need even more, namely divisibility by $f^4$.

Therefore we will split our special automorphism into finitely many factors, each of them being equal to the identity on $\{f = 0\}$. Moreover, we will demand that their difference to identity to be divisible by $f^4$. Another obvious necessary condition for the further splitting of these automorphisms $G$ into products of unipotent automorphisms is that they are null-homotopic. For our topological solution to that problem, we even have to demand that the $G$'s are \textit{strongly null-homotopic}. By this we mean that on $\{f = 0\}$ the whole homotopy is identical to $G_t = \operatorname{Id}$, $\forall t \in [0,1]$. Observe that $G_0$ and $G_1$ are equal to the identity on $\{f = 0\}$ by definition.

\begin{Thm}\label{product of suitable}

Let $X$ be a Stein space and $f_i \in \mathcal{O}(X)$, $i= 1,2, \ldots, m$ finitely many holomorphic functions without common zeros. Then every null-homotopic special holomorphic vectorbundle automorphism $F \in \SAut(E)$ can be written as a product 
\[F = \prod_{i=1}^m G_i,
\]
where each of the $G_i \in \SAut (E)$ is a special holomorphic vectorbundle automorphism, whose difference to the identity is divisible by $f_i^4$.
Moreover, the $G_i$'s are  strongly null-homotopic, i.e., they are null-homotopic such that on $\{f_i = 0\}$ the homotopy $(G_i)_t = \operatorname{Id}$, $\forall t \in [0,1]$.

\end{Thm}

\subsection{The Topological Solution}
\label{ssection:topologicalsuitablepairs}

\begin{Thm} \label{product of suitable continuous}

Let $X$ be a Stein space and $f_i \in \mathcal{O}(X)$, $i= 1,2, \ldots, m$ finitely many holomorphic functions without common zeros. Then every null-homotopic special  continuous vectorbundle automorphism $F \in \SAut(E)$ can be written as a product 
\[F = \prod_{i=1}^m G_i,
\]
where each of the $G_i \in \SAut(E)$ is a special continuous vectorbundle automorphism such that $G_i (x) = \operatorname{Id_{E_x}}$ when $f_i (x) = 0$.  Moreover, each $G_i$ is strongly null-homotopic, i.e., null-homotopic in such a way that on the set $f_i (x) = 0$ the whole homotopy is constant $\operatorname{Id}_{E_x}$. 

\end{Thm}

\begin{proof}

Given a null-homotopic section $F$ of the fibre bundle $\SAut(E)$ over $X$, we will construct continuous sections $G_{i}$ of the same fibre bundle such that their product is equal to $F$ and $G_{i}|_{A_{i}} = \operatorname{Id}_{E_x}$, where $A_{i} = \{f_{i} = 0\}$ and $f_{i}$ are holomorphic functions on $X$ without common zeros. Thus $\cap A_{i} = \emptyset$. 


Since $F$ is null-homotopic, there exists a continuous map $F_{t} : X \rightarrow \SAut(E)$, $\forall t \in [0,1]$ such that $F_{0} = \text{Id}$ and $F_{1} = F$.

First consider the case $m=2$. Let $U_{i}$ be open neighbourhoods of $A_{i}$ in $X$. Since $A_1 \cap A_2 = \emptyset$, we can define a smooth cutoff function $\chi : X \rightarrow [0,1]$ such that $\chi \equiv 0$ on $U_1$ and $\chi \equiv 1$ on $U_2$. Now we put 
\[\alpha(x) = F_{\chi(x)}(x),\ \beta(x) = F_{\chi(x)}^{-1}F(x),\ \forall x \in X.
\]
It is obvious that $\alpha \cdot \beta = F$. At the same time, $\alpha|_{A_1} = F_0 = \operatorname{Id}$ and $\beta|_{A_2} = F_1^{-1} \cdot F = F^{-1} \cdot F = \operatorname{Id}$. Bare in mind that $\alpha$ and $\beta$ are just $G_1$ and $G_2$ that we were looking for.

Now to prove that $\alpha$ and $\beta$ are null-homotopic in such a way that on the set $A_1$ and $A_2$ the whole homotopy is constant $\operatorname{Id}$, respectively, we just place a $t$ in front of $\chi$:
\[\alpha_t(x) = F_{t\chi(x)}(x), \ \beta_t(x) = F_{t\chi(x)}^{-1}F_t(x), \ \forall x \in X.
\]
We easily see that \[\alpha_0 = F_0 = \operatorname{Id},\ \beta_0 = F_0^{-1} F_0 = \operatorname{Id},\]
\[\alpha_1 = F_{\chi} = \alpha,\ \beta_1 = F_\chi^{-1}F_1 =  F_\chi^{-1}F = \beta,
\]
\[\alpha_t|_{A_1} = F_0 = \operatorname{Id},\ \beta_t|_{A_2} = F_t^{-1} \cdot F_t = \operatorname{Id}.
\]
In addition, we have obtained that $\alpha_t \beta_t = F_t$ on $X$.

Now we proceed with the case $m=3$. Let $U$ be an open neighbourhood of $A_1 \cap A_2$. Then on $Y:=X \setminus \overline{U}$, the sets $A_1$ and $A_2$ have empty intersection, so we can apply the previous step. Thus, we obtain null-homotopic special continuous vectorbundle automorphisms $\widehat{\alpha}$ and $\widehat{\beta}$, which are constant the identity on $A_1 \cap Y$ and $A_2 \cap Y$, respectively. At the same time, we have the homotopies $\widehat{\alpha}_t$ and $\widehat{\beta}_t$ such that $\widehat{\alpha}_t = \operatorname{Id}$ on $A_1 \cap Y$, $\widehat{\beta}_t = \text{Id}$ on $A_2 \cap Y$ and $\widehat{\alpha}_t \widehat{\beta}_t = F_t$ on $Y$. We simply put $\widehat{\gamma}_t := \text{Id}$ and $\widehat{\gamma}:= \widehat{\gamma}_0$ on $Y$. It is obvious that $\widehat{\alpha} \widehat{\beta} \widehat{\gamma} = F$ on $Y$ and $\widehat{\alpha}$, $\widehat{\beta}$ and $\widehat{\gamma}$ have the desired properties. In addition, $\widehat{\alpha}_t \widehat{\beta}_t \widehat{\gamma}_t = F_t$ on $Y$.

Consider a smooth cutoff function $\chi : X \rightarrow [0,1]$ such that $\chi \equiv 0$ on $\overline{U}$ and $\chi \equiv 1$ on $X \setminus U_0$, where $U_0 \supset U$ is an open neighbourhood of $U$ such that $U_0 \cap A_3 = \emptyset$. Define $\alpha, \beta : X \rightarrow \SAut(E)$ by
\[\alpha(x) = \left\{\begin{array}{rl}
\widehat{\alpha}_{\chi(x)}, & x \in X \setminus U \\
\operatorname{Id}, & x \in U
\end{array}\right.,\ \beta(x) = \left\{\begin{array}{rl}
\widehat{\beta}_{\chi(x)}, & x \in X \setminus U \\
\operatorname{Id}, & x \in U
\end{array}\right.,
\]
respectively. We put $\gamma = \beta^{-1} \alpha^{-1} F$. It is clear that $\alpha \beta \gamma = F$ on $X$ and $\gamma = \operatorname{Id}$ on $A_3$ because $F = F_1 = \widehat{\alpha}_1 \widehat{\beta}_1 = \alpha \beta$ on $A_3$. Here we have used the fact that $U_0 \cap A_3 = \emptyset$. In this case, $\alpha$, $\beta$ and $\gamma$ are $G_1$, $G_2$ and $G_3$ from the conclusion of the theorem.

Now, as before, by placing a $t$ in front of $\chi$ we construct the homotopies
\[\alpha_t(x) = \left\{\begin{array}{rl}
\widehat{\alpha}_{t\chi(x)}, & x \in X \setminus U \\
\text{Id}, & x \in U
\end{array}\right., \ \beta_t(x) = \left\{\begin{array}{rl}
\widehat{\beta}_{t\chi(x)}, & x \in X \setminus U \\
\text{Id}, & x \in U
\end{array}\right.,
\]
respectively. We also define $\gamma_t = \beta_t^{-1} \alpha_t^{-1} F_t$. 

We  see that by definition, $\alpha_0 = \beta_0 = \gamma_0 = \operatorname{Id}$ on $U$. Outside $U$, we are in the case $m=2$ and we have that $\alpha_0 = \widehat{\alpha}_0 = \operatorname{Id}$, $\beta_0 = \widehat{\beta}_0 = \operatorname{Id}$ and $\gamma_0 = F_0 = \operatorname{Id}$. It is obvious that $\alpha_1 = \alpha$, $\beta_1 = \beta$ and $\gamma_1 = \beta_1^{-1} \alpha_1^{-1} F_1 = \beta^{-1} \alpha^{-1} F = \gamma$.

Following a similar argument to the one before, $\alpha_t|_{A_1} = \operatorname{Id}$, $\beta_t|_{A_2} = \operatorname{Id}$ and $\gamma_t|_{A_3} = (\beta_t^{-1} \alpha_t^{-1} F_t)|_{A_3} = \operatorname{Id}$. Here we have used the fact that $\alpha_t \beta_t = F_t$ from the case $m=2$. Again, it is clear that $\alpha_t \beta_t \gamma_t = F_t$ on $X$.

Now for the general case. Without loss of generality, we can assume that $A_1 \cap \cdots \cap A_{m-1} \neq \emptyset$ (otherwise we simply set $G_m = \operatorname{Id}$ together with the homotopy $(G_m)_t = \operatorname{Id}$). Then we apply the same procedure as in the previous case considering an open neighbourhood $U$ of this intersection, defining $Y := X \setminus \overline{U}$ and so on. 

\end{proof}

\begin{Rem} According to Proposition \ref{existence of pairs}, in the previous theorem, we need $m \le n+1$, where $n$ is the dimension of the Stein space $X$.
    
\end{Rem}

\subsection{Application of the Oka Principle for Oka Pairs}
\label{ssection:Oka pairs}

We will apply Studers Theorem \ref{Thm:Studer}. Let us specify the sheaves. As before $f_i \in \mathcal{O}(X)$, $i= 1,2, \ldots, m$ are  finitely 
many holomorphic functions without common zeros and let $A_i$ denote their zero sets $A_i := \{x \in X : f_i (x) = 0\}$.

We define sheaves $\mathcal{F}_i$ by specifying the presheaves for any $U$ open in $X$ to be 
\[\mathcal {F}_i (U) := \{\alpha \in \SAut_{\text{hol}}(E\vert_U) : f_i \mid (\alpha  - \operatorname{Id})\},
\]
and similarily sheaves $\mathcal{G}_i$ by 
\[\mathcal {G}_i (U) := \{\alpha \in \SAut_{\text{cont}}(E\vert_U) : \alpha\vert_{U \cap A_i} = \operatorname{Id})\}.
\]

The sheaf $\Phi$ is defined as a presheaf by 
\[\Phi (U) := \{(\alpha_1, \alpha_2, \ldots \alpha_m) \in \mathcal {F}_1 (U)\times \mathcal {F}_2 (U) \times \cdots \times \mathcal {F}_m(U) : \alpha_1 \circ \alpha_2 \circ \ldots \circ \alpha_m = f\vert_U\},
\]
and similarily the sheaf $\Psi$ as 
\[\Psi (U) :=\{ (\alpha_1, \alpha_2, \ldots \alpha_m) \in \mathcal {G}_1 (U)\times \mathcal {G}_2 (U) \times \cdots \times \mathcal {G}_m(U) : \alpha_1 \circ \alpha_2 \circ \ldots \circ \alpha_m = f\vert_U \}.
\]

\begin{Lem} \label{Lemma:weak} 

The inclusion of sheaves $\Phi \hookrightarrow \Psi$ is a weak homotopy equivalence.

\begin{proof} Let $x_0\in X$. Since the $f_i$'s have no common zeros, there exists an integer $k$, $1 \le k \le m$ with $f_k (x_0) \ne 0$. Thus there is an open neighborhood $U = X \setminus A_k$ of $x_0$ such that $f_k$ is nowhere zero on $U$. On this neighborhood $U$ the restriction of the sheaf $\mathcal{F}_k$ is isomorphic to the restriction of the sheaf $\SAut_{\text{hol}}$. Hence by expressing 
\begin{equation}
\alpha_k (x) = \alpha_{k-1}^{-1} (x) \circ \cdots \circ \alpha_1^{-1} \circ f(x) \circ \alpha_{m}^{-1} (x) \circ \cdots \circ \alpha_{k+1}^{-1}(x) \label{formula:express} 
\end{equation}
the restriction of $\Phi$ to $U$ is isomorphic to the sheaf $\mathcal{A}_k$ defined for an open set $V\subset U$ by 
\[\mathcal{A}_k (V) := \{(\alpha_1, \alpha_2, \ldots, \widehat{\alpha}_k, \ldots, \alpha_m) \in \mathcal {F}_1 (V) \times \mathcal{F}_2 (V) \times \cdots \times \widehat{\mathcal{F}}_k(V) \cdots \times \mathcal {F}_m (V)\},
\]
and analogously by \eqref{formula:express} the restriction of $\Phi$ to $U$ is isomorphic to the sheaf $\mathcal{B}_k$ defined by
\[\mathcal{B}_k(V) := \{(\alpha_1, \alpha_2, \ldots, \widehat{\alpha}_k, \ldots, \alpha_m) \in \mathcal{G}_1(V) \times \mathcal{G}_2(V) \times \cdots \times \widehat{\mathcal{G}}_k(V) \times \cdots \times \mathcal {G}_m(V)\}.
\]
 
The sheaf $\mathcal{A}_k$ is a coherent subsheaf of the product sheaf $\SAut_{\text{hol}}(E)^{(k-1)}$ in the sense of Forster-Ramspott (see the definitions in \cite{Okasche Paare} 2.1. on pp. 265, 266). In other words, if we denote by $\exp : \operatorname{End}^0(E) \to \SAut(E)$ the exponential map from trace zero endomorphisms (the Lie algebra bundle) to  special automorphisms (the Lie group bundle), then sections in a neighborhood of zero in the coherent subsheaf 
\[f_1 \operatorname{End}^0_{\text{hol}}(E)\times \cdots \times \widehat{f_{k} \operatorname{End}^0_{\text{hol}}(E)} \times \cdots \times f_m \operatorname{End}^0_{\text{hol}}(E)^{m}
\]
are mapped by the exponential map (which in this case is the ($k-1$)-fold product of the above described exponential map) to a neighborhood of the identity in $\mathcal{A}_k \subset \operatorname{Aut}_{\text{hol}}(E)^{(k-1)}$.

The sheaf $\Psi$ is defined as a subsheaf of $\SAut_{\text{cont}}^k(E)$ exactly in such a way as to satisfy the pointwise homotopy condition (PH) of Forster and Ramspott (see \cite{Okasche Paare} Beispiel b) on p. 268). Here we have that $\Psi(x) = \Phi(x)$.
The expressions $\Psi (x)$ and $\Phi (x)$ have the meaning of the reachable points of these sheaves. For definition of reachable points see \cite{Okasche Paare} middle of p. 267 or the discussion before Proposition \ref{ellipticsubmersion} in our paper.

The  equality of reachable points also holds for the pair of sheafs $\mathcal{A}_k \hookrightarrow \mathcal{B}_k$. The weak homotopy equivalence for the inclusion of these sheafs is exactly Forster-Ramspott's condition (RH.0) proved on p. 275 (end) up to p. 276 (middle). It is a consequence of Lemma 1 loc. cit. on p. 269.

\end{proof}

\end{Lem}

\begin{Rem}

The algebra/group structure in \cite{Okasche Paare} is overkill. A dominating spray restricted to a coherent subsheaf of a vector bundle  would have been sufficient (see Studer's Ph.D. Thesis Prop. 4.1). The exponential map is just a particular instance of a spray.

\end{Rem}
 
By using cutoff functions (partition of unity) it follows:
 
\begin{Lem} 

The sheaf $\Psi$ satisfies condition (2) of Theorem \ref{Thm:Studer}.
 
\end{Lem}
 
\begin{Lem} 

The sheaf $\Phi$ satisfies condition (1) of Theorem \ref{Thm:Studer}.

\end{Lem}

\begin{proof}

We use the notation for weakly flexible pair $(A,B)$ as in  Definition  \ref{Def:weaklyflexible} from Section \ref{section:preliminaries}. We can assume that $B$ is contained in $U = X \setminus A_k$ as in the proof of Lemma \ref{Lemma:weak}. The proof proceeds as usually in $2$ steps, approximation and gluing of nearby sections.
 
The approximation part, i.e., bringing the (parametrized) section $b$ as close as we wish to $a$ in a neighborhood of $A \cap B$, is achieved by approximation of the homotopy $c_s$ connecting $a$ and $b$ over (a neighborhood of) $A \cap B$ by a homotopy $b_s$ over (a neighborhood of) $B$ in a coherent sheaf setting. More precisely, as explained in the proof of Lemma \ref{Lemma:weak}, the section $b$ is a section of the coherent subsheaf (in the sense of Forster-Ramspott) of the product sheaf $\SAut_{\text{hol}}(E)^{(k-1)}$. Thus approximation is done in the proof of Satz 2 starting at p. 275. The relevant part for approximation in the proof of Satz 2 is the proof of b) in the middle of p. 276, which is based on Lemma 2 on p. 271. Lemma 2 is the case of sections close to identity. One reduces to that case by dividing the homotopy into finitely many steps, so that the "difference" (= quotient of automorphisms) between two steps is then near to the identity.

In order to glue the (parametrized) sections $a$ and $b_1$, which are close over (a neighborhood) of $A \cap B$ into a global section over (a neighborhood) of $A \cup B$, we use the following "thickening" (also called a \textit{local spray} by Forstneri\v c):

The (fibrewise) map $\SAut(E)^{m-1} \times \SAut(E)^{m}$ given by
\begin{multline}
(\beta_1(x), \beta_2(x), \ldots, \beta^{m-1}(x)) \times (\alpha_1 (x), \alpha_2 (x), \ldots, \alpha_k (x)) \\
\mapsto (\alpha_1 (x) \beta_1(x)^{-1}, \beta_1(x) \alpha_2 (x) \beta_2 (x)^{-1}, \ldots, \beta_{l-1} (x) \alpha_l (x) \beta_l(x)^{-1}, \ldots, \beta_{m-1} (x) \alpha_m (x)) 
\end{multline}
is a "thickening" for the sub sheaf $\mathcal{H}_f$ of holomorphic sections of $\SAut (E)^m$ defined by 
\[\mathcal{H}_f (U) := \{(\alpha_1, \alpha_2, \ldots, \alpha_m) : \alpha_1 \circ \alpha_2 \circ \ldots \circ \alpha_m = f\vert_U\}.
\]

The subsheaf $\mathcal{J}$ of $\SAut(E)^{m-1}$ given by 
\begin{multline}
\mathcal{J} = \{(\beta_1, \beta_2, \ldots, \beta^{m-1}) \in \SAut(E)^{m-1}:
\\ f_1 \mid (\beta_1 - \operatorname{Id}), f_2 \mid (\beta_1^{-1}\beta_2-\operatorname{Id}), \ldots, f_{m-1} \mid (\beta_{m-2}^{-1} \beta_{m-1}-\operatorname{Id}), f_m \mid (\beta_{m-1} - \operatorname{Id})\}
\end{multline}
leaves the sheaf $\Phi$ invariant.

Define the subsheaf $\text{Lie}(\mathcal{J})$ of $\End^0(E)^{m-1}$ corresponding to $\mathcal{J}$ as
\begin{multline} 
\text{Lie}(\mathcal{J}) = \{(v_1, v_2, \ldots, v_{m-1}) \in \End^0(E)^{m-1}: \\ f_1 \mid v_1, f_2 \mid (v_1-v_2), \ldots, f_{m-1} \mid (v_{m-2} - v_{m-1}), f_m \mid v_{m-1}\},
\end{multline}
which is a coherent subsheaf of $\End^0(E)^{m-1}$. 

The map 
\begin{multline} \label{localspray}
v = (v_1, v_2, \ldots, v_{k-1}) \mapsto a(x) \exp(v) \\
:= (\alpha_1(x)\exp(-v_1), \exp(v_1)\alpha_2(x)\exp(-v_2), \ldots, \\ \exp(v_{l-1})\alpha_l(x)\exp(-v_{l)}, \ldots, \exp(v_{m-1})\alpha_k(x)
\end{multline}
is the required (even global) spray around the section $a = (\alpha_1, \alpha_2, \ldots, \alpha_m) \in \Phi$.

The transition function now comes from $\log(a(x)\exp(v(x)))$ and the gluing
can be achieved by Corollary 3.1. in \cite{Studer1}. Alternatively, by the
method of Cartan, the gluing is done in Lemma 3 on p. 273 in \cite{Okasche Paare}.

\end{proof}

\begin{Lem} 

The sheaf $\Phi$ is complete.

\begin{proof}

A subsheaf of a free $\mathcal{O}$-module sheaf is complete if and only if it is coherent (see \cite{Okasche Paare} Hilfssatz 1 on p. 266). Since a neighborhood of any section of $\Phi$ is parametrized by the sections near zero of the coherent sheaf $\text{Lie}(\mathcal{J})$ (see equation \ref{localspray}) the assertion follows.

\end{proof}

\end{Lem}

\begin{proof}[Proof of Theorem \ref{product of suitable}]

 First remark that if the holomorphic functions $f_1, f_2, \ldots f_m$ have no common zeros, so do their fourth powers $f_1^4, f_2^4, \ldots f_m^4$ and $f_i(x) = 0$ if and only if $f_i^4 (x) = 0$. Theorem \ref{product of suitable continuous} provides a global section $\psi \in \Gamma(\Psi,X)$ of
 the sheaf $\Psi$. If we use the sheaf $\Phi$ with $f_i$ replaced by $f_i^4$, we conclude from Theorem \ref{Thm:Studer} that $\psi$ is homotopic (by global sections of $\Psi$) to a global section $\phi \in \Gamma(\Phi, X)$ of $\Phi$. Let $G_1, \ldots, G_m$ be  the components of this holomorphic section $\phi$. They satisfy all demanded properties. Firstly, $\phi$ is null-homotopic, and thus its components. This is because $\psi$ was null-homotopic by construction and $\phi$ is homotopic to $\psi$. Secondly, the components of $\psi$ are strongly null-homotopic and since the homotopy from $\psi$ to $\phi$ is by sections of the sheaf $\Phi$, the same holds for the components $G_i$ of $\psi$.
 
\end{proof} 

\section{Factorization of Suitable Automorphisms into Unipotent Factors}
\label{section:factorizationofsuitableautomorphisms}

Let $\pi : E \rightarrow X$ be a holomorphic vector bundle of rank $2$, where $X$ is a Stein space of dimension $n$. We are given a null-homotopic special holomorphic automorphism $F$ of $E$.

By Proposition \ref{existence of pairs} we have $n+1$ pairs $N_i^+, N_i^-$ together with the holomorphic functions $f_i$ constructed in Lemma \ref{Lemma:pair}, such that the $f_i$ have no common zeros. Applying Theorem \ref{product of suitable}, we can factorize $F$ as a a product 
\[F = \prod_{i=1}^m G_i,
\]
where each of the $G_i \in \SAut(E)$ is a special holomorphic vectorbundle automorphism, whose difference to the identity is divisible by $f_i^4$.
Moreover, the $G_i$'s are  strongly null-homotopic, i.e., they are null-homotopic such that on $\{ f_i=0 \}$ the homotopy $(G_i)_t = \operatorname{Id}$, $\forall t \in [0,1]$.

The goal of this section is to prove that each of the automorphisms $G_i$ can be written further as a finite product of replicas of the corresponding unipotent pair 
\[G_i = U_i^- (z_1) U_i^+(z_2) \cdots U_i^- (z_{2k-1}) U_i^+(z_{2k}),
\]
where the number $k$ may depend on $i$.

From now on we concentrate on this problem and recall the local setting from Section \ref{section:pairs}. In some trivialization $E \simeq U_{i} \times \mathbb{C}^2$, the sections $s_1, s_2$ of $E$ are represented by vectors $s_1^{i}(x), s_2^{i}(x) \in \C^2$ which may become linearly dependent in some points $x$. We denote by $S_i(x)$ the matrix which has the above sections on its columns. We have that 
\[N^+(x) = S_i(x) \cdot\left(\begin{array}{cc} 0 & f(x) \\ 0 & 0 \end{array}\right)\cdot S_i^{-1}(x) = S_i(x) \cdot\left(\begin{array}{cc} 0 & f_i(x) \\ 0 & 0 \end{array}\right)\cdot S_i^{\#}(x)
\]
and
\[N^-(x) = S_i(x) \cdot\left(\begin{array}{cc} 0 & 0 \\ f(x) & 0 \end{array}\right)\cdot S_i^{-1}(x) = S_i(x) \cdot\left(\begin{array}{cc} 0 & 0 \\ f_i(x) & 0 \end{array}\right)\cdot S_i^{\#}(x),
\]
where $S_i^{\#}(x)$ is the adjoint of $S_i(x)$ and $f_i  = \dfrac{f}{\det S_i} \in \O (U_i)$ (see Section \ref{section:pairs}). Remember that $U^{\pm} = \operatorname{Id} + N^{\pm}$ and $U^{\pm}(h) = \operatorname{Id} + h\cdot N^{\pm}$. 
Our main result, namely Theorem \ref{Thm:main} will then immediately follow from the following result.

\begin{Thm} \label{factorization-into-replicas}

 Let $U^+, U^-$ and $f \in \O(X)$ be as above and let $G \in SAut(E)$
 with the properties that $f^4 \mid G - \operatorname{Id}$ and $G$ is strongly null-homotopic. Then 
 \[G = U^-(h_1) \cdot U^+(h_2) \cdot \ldots \cdot U^-(h_{2k-1}) \cdot U^+(h_{2k})
 \]
 for some integer $k$ and holomorphic functions $h_i \in \O(X)$, $i = 1, 2, \ldots, k$.
 
\end{Thm}  

The proof of this theorem is contained in Subsection \ref{ssection: holomorphicULfactorization}. As a preparation, in Subsection \ref{ssection:stratified}, we prove that for a certain fibration, the Oka principle for stratified elliptic submersions, namely Theorem \ref{t:forstneric}, is applicable. In Subsection \ref{ssection:topologicalULfactorization} we provide a topological solution for our problem.

\subsection{Preparing for the Oka Principle for Stratified Elliptic Submersion}
\label{ssection:stratified}

We start with some preliminary considerations. Consider an automorphism $G$ with the property that $G - \operatorname{Id}$ is divisible by $f^3$.
The equation that we would like to solve is 
\begin{equation}
\label{factorize1} 
G(x) =  U^-(h_1(x)) \cdot U^+(h_2(x)) \cdot \ldots \cdot U^+(h_{2n}(x)), 
\end{equation}
where $h_i \in \O(X)$. 

In other words we would like to show that there exists a holomorphic mapping $g=(z_1, \ldots, z_{2n}) : X \rightarrow \mathbb{C}^{2n}$ such that the following diagram commutes:
\[\begin{xy} 
(0,0)*+{X}="a";(35,0)*+{\SAut(E)}="b";(35,20)*+{\mathbb{C}^{2n}}="c";
{\ar^-{G} "a";"b"};
{\ar^-{\psi_{2n}} "c";"b"};
{\ar@{-->}^-{g} "a";"c"};
 \end{xy}\] 
where $\psi_{2n}:\mathbb{C}^{2n} \rightarrow \SAut(E)$ is given by
\[(z_1, \ldots, z_{2n}) \longmapsto U^-(z_1) \cdot U^+(z_2) \cdot \ldots U^-(z_{2n-1}) \cdot U^+(z_{2n}).
\]
This is equivalent to finding a holomorphic section of the pullback of the bundle $\xi = (\mathbb{C}^{2n}, \psi_{2n}, \SAut(E))$ by $G$. Recall that the pullback bundle $G^{\ast}\xi = (G^{\ast}(\mathbb{C}^{2n}), G^{\ast}\psi_{2n}, X)$, where
\[G^{\ast}(\mathbb{C}^{2n}) = \{(x,z) \in X \times \mathbb{C}^{2n} : G(x) = \psi_{2n}(z)\} \subset X \times \mathbb{C}^{2n}
\]
and the projection is $G^{\ast}\psi_{2n}(x,z) = x$. We also define $G_\xi : G^{\ast}(\mathbb{C}^{2n}) \rightarrow \mathbb{C}^{2n}$ by $G_\xi(x,z) = z$. We obtain the following commutative diagram, so that a section of the left side bundle is equivalent to a lift $g$ in the diagram above.
\begin{equation} \label{diagram}    
\begin{xy}
(0,0)*+{X}="a";(35,0)*+{\SAut(E)}="b";(35,20)*+{\mathbb{C}^{2n}}="c";(0,20)*+{G^{\ast}(\mathbb{C}^{2n})}="d";
{\ar^-{G} "a";"b"};
{\ar^-{\psi_{2n}} "c";"b"};
{\ar^-{G_\xi} "d";"c"};
{\ar_-{G^{\ast}\psi_{2n}} "d";"a"};
\end{xy}
\end{equation}

The following technical lemma will be used to understand the restriction of the fibration $G^{\ast}\xi$ to $\{f = 0\}$.

\begin{Lem}\label{modulof3} 

Define the matrix 
\[\left(\begin{array}{cc} Q^k_{11} (z,f)& Q^k_{12} (z,f)
\\ Q^k_{21} (z,f)& Q^k_{22} (z,f) \end{array}\right) :=\left(\begin{array}{cc} 1 & z_2 f \\ 0 & 1 \end{array}\right) \cdot \left(\begin{array}{cc} 1 & 0 \\ z_3 f & 1 \end{array}\right) \cdot \ldots \cdot 
\left(\begin{array}{cc} 1 & z_{2k} f \\ 0 & 1 \end{array}\right) \cdot \left(\begin{array}{cc} 1 & 0 \\ z_{2k+1} f & 1 \end{array}\right).
\] 
Then 
\[Q^k_{11}(z,f) = 1 + f^2 \sum_{1\le i\le j\le k} z_{2i} z_{2j+1} + f^3 \widetilde{Q}^k_{11}(z,f).
\]

\begin{proof}
    
A straightforward induction shows not only the required formula (\ref{modulof3}), but also the formulas for the other entries modulo $f^3$:
\[Q^k_{12}(z,f) = f \sum_{1\le i\le k} z_{2i} + f^3 \widetilde{Q}^k_{12} (z,f),
\]
\[Q^k_{21}(z,f) = f \sum_{1\le  j\le k} z_{2j+1} + f^3 \widetilde{Q}^k_{21} (z,f),
\]
\[Q^k_{22}(z,f) = 1 + f^2 \sum_{1\le j < i\le k} z_{2i} z_{2j+1} + f^3 \Tilde{Q}^k_{22}(z,f).
\]

\end{proof}

\end{Lem}

The next result is crucial for applying the Oka principle for stratified elliptic submersions. 

Recall that the set of holomorphically reachable points $\text{A}(b)$ in a fibre $A_b = \varphi^{-1}(b)$ of a fibration (meaning holomorphic map) $\varphi: A \to B$ over $b\in B$ is given by 
\[\text{A}(b) = \bigcup_{U \subseteq_{\rm{op}} X,\, b \in U,\, s \in \Gamma_{\text{hol}}(U, \varphi)}s(b)
\]
of the values of sections $\Gamma_{\text{hol}}(U,\varphi) = \{s:B \to A : s\, \text{is}\, \text{holomorphic},\, \varphi \circ s = \operatorname{Id}_B\}$. 
The set of holomorphically reachable points of a fibration is the union of the holomorphically reachable points of all fibres. For a holomorphic map one can speak of holomorphically or continuously reachable points. For the fibration for which we are using this notion, these two sets will coincide.

\begin{Prop}\label{ellipticsubmersion}

Let $f \in \O(X)$ be a holomorphic function on a complex space $X$ and let $n \ge 2$. Given a holomorphic map $G : X \to \SAut(E)$ with the property that $G - \operatorname{Id}$ is divisible by $f^3$, then the reachable points of the fibration $(G^{\ast}(\mathbb{C}^{2n}) \setminus \operatorname{Sing}, G^{\ast}\psi_{2n}, X)$ form a stratified elliptic submersion.

\end{Prop}

Here $\operatorname{Sing}$ denotes the set of singular points of the fibration $G^{\ast}\xi$ which is given by 
\[\operatorname{Sing} = \{(x,z) \in X \times \mathbb{C}^{2n} : G(x) = \psi_{2n}(z)\, \text{and}\, z_2 = z_3 = \ldots = z_{2n-1} = 0\}.
\]

\begin{proof}
    
The equation that we would like to solve and which describes the fibration $G^{\ast}\xi$ is 
\begin{equation}
\label{factorize} 
G(x) =  U^- (z_1(x)) \cdot U^+ (z_2(x)) \cdot \ldots \cdot U^+ (z_{2n}(x)), 
\end{equation}
for  $z_i \in \O(X)$. 

By construction the vector bundle $E$ is trivial over $X \setminus\{f=0\}$ (trivialized by the sections $s_1, s_2$). Therefore the map $G$ can be considered as a map $G: X \to \operatorname{SL}_2(\C)$ and the pair of unipotent automorphisms is given by
\[U^+ = \left(\begin{array}{cc} 1 & f(x) \\ 0 & 1 \end{array}\right) 
\]
and 
\[U^- = \left(\begin{array}{cc} 1 & 0 \\ f(x) & 1 \end{array}\right).
\]
Our problem is therefore equivalent on $X \setminus\{f = 0\}$ to the standard factorization problem solved by Ivarsson and the second author in \cite{IK2}. There, all the required claims of Theorem \ref{Thm:main} are proved. In the aforementioned paper, the entry $f(x)$ is actually constant $1$. However, by rescaling the restriction of the fibration to $X \setminus \{f = 0\}$, one finds the globally integrable vector fields along the fibres spanning the tangent space at every point. Remark that on each fibre, the number $f(x)$ is a nonzero constant and globally integrable fields stay globally integrable when multiplied by a constant. The stratification in \cite{IK2} is given by 
\[X_0 = X \setminus \{f = 0\} \supset X_1 = \{a = 0\} \supset \emptyset,
\]
where $a \in \O(X)$ is the left upper entry in the matrix 
\[G (x) = \left(\begin{array}{cc} a(x) & b(x) \\ c(x) & d(x) \end{array}\right).
\]
We refer to the paper \cite{IK2} for the details. 

In our case we define the stratification as 
\[X_0 = X \subset X_1 = (\{f = 0\} \cup \{a = 0\}) \supset X_2 = \{f = 0\} \supset \emptyset.
\]
Observe that on $\{f = 0\}$ the automorphism $G$ is equal to the identity. Thus, in our local trivialization the entry $a$ of the matrix tends to $1$ when approaching the analytic subset $\{f = 0\}$ of $X$. Therefore $\{a = 0\}$ is a closed analytic subset not only of $X \setminus \{f = 0\}$, but also of $X$. Hence, $X_1 = (\{f = 0\} \cup \{a = 0\})$ is a closed analytic subset of $X$.

Assuming the results from \cite{IK2}, now we only have to analyze our fibration on $\{f = 0\}$. Since being a stratified elliptic submersion is a local matter (except for the stratification), we go into the local setting   described before Proposition \ref{ellipticsubmersion}. On $U_i$, outside the analytic subset where the sections $s_1, s_2$ are linearly dependent, equation \ref{factorize} is equivalent to
\[S_i(x)^{-1}\operatorname{Mat}(G)(x)S_i(x) =  S_i(x)^{-1}U^- (z_1(x))S_i(x) \cdot S_{i}^{-1}(x)U^+ (z_2(x))S_i(x) \cdot \ldots,
\]
where $\operatorname{Mat}(G)(x) \in \operatorname{SL}_2(\C)$ is the matrix representing the automorphism $G(x)$ in local coordinates. Since $G - \operatorname{Id}$ is divisible by $f^3$, we will write 
\[\operatorname{Mat}(G)(x) = \operatorname{Id} + f^3(x) \left(\begin{array}{cc} a_1(x) & b_1(x) \\ c_1(x) & d_1(x) \end{array}\right).
\]
Hence
\[\operatorname{Id} + f^3(x) S_i^{-1}(x) \left(\begin{array}{cc} a_1(x) & b_1(x) \\ c_1(x) & d_1(x) \end{array}\right) S_i (x) = \left(\begin{array}{cc} 1 & 0 \\ z_1f(x) & 1 \end{array}\right) \cdot \left(\begin{array}{cc} 1 & z_2f(x) \\ 0 & 1 \end{array}\right) \cdot \ldots,
\]
or equivalently

\begin{equation}
\label{factorizelocal}    
\begin{split}
\operatorname{Id} + f^2(x) f_i (x) S_i^{\#}(x) & \left(\begin{array}{cc} a_1(x) & b_1(x) \\ c_1(x) & d_1(x) \end{array}\right) S_i (x) = \\ & = \left(\begin{array}{cc} 1 & 0 \\ z_1f(x) & 1 \end{array}\right) \cdot \left(\begin{array}{cc} 1 & z_2f(x) \\ 0 & 1 \end{array}\right) \cdot \ldots. 
\end{split}
\end{equation}

Since both sides of this last equation are holomorphic on all of $U_i$, the identity principle for holomorphic/continuous maps ($X \setminus \{f = 0\}$ is dense in $X$) shows that our problem (\ref{factorize}) is equivalent to this last equation, namely equation (\ref{factorizelocal}), on all of $U_i$ and not only where the the sections are linearly independent. 

We introduce the notation
\[\left(\begin{array}{cc} a(x) & b(x) \\ c(x) & d (x) \end{array}\right) := f_i (x) S_i^{\#}(x)\left(\begin{array}{cc} a_1(x) & b_1(x) \\ c_1(x) & d_1(x) \end{array}\right) S_i(x).
\]
Now the reachable points of $G^{\ast}(\xi)$ are described by 
\begin{equation}
\label{factorizelocal1}    
\operatorname{Id} + f^2(x) \left(\begin{array}{cc} a(x) & b(x) \\ c(x) & d (x) \end{array}\right) = \left(\begin{array}{cc} 1 & 0 \\ z_1f(x) & 1 \end{array}\right) \cdot \left(\begin{array}{cc} 1 & z_2f(x) \\ 0 & 1 \end{array}\right) \cdot \ldots . 
\end{equation}
To study this equation further, we bring the first and the last matrix from the right to the left hand side. We obtain

\begin{equation*}
\begin{split}
\left(\begin{array}{cc} 1 & 0 \\- z_1 f & 1 \end{array}\right) &
 \left(\begin{array}{cc} 1+ f^2 (x) a(x) & f^2 (x) b(x) \\  f^2 (x) c(x) & 1 + f^2(x)d(x) \end{array}\right)
 \left(\begin{array}{cc} 1 & -z_{2n} f \\ 0 & 1 \end{array}\right)
 = \\ & = \left(\begin{array}{cc} 1 & z_2f(x) \\ 0 & 1 \end{array}\right) \cdot \ldots 
 \cdot \left(\begin{array}{cc} 1 & 0 \\ z_{2n-1}f(x) & 1 \end{array}\right). 
\end{split}
\end{equation*}
With the notations of Lemma (\ref{modulof3}) we can write this in the following way:

 \begin{equation}
 \begin{split}
\label{factorizelocal2} & \left(\begin{array}{cc} 1 + f^2(x)a(x) & f^2 (x) b(x) - z_{2n}(1+ f^2(x)a(x))  \\  f^2 (x) c(x) - z_1 (1+f^2(x)a(x)) & \ast \end{array}\right) = \\ & \phantom{abcdefghijklmn} = \left(\begin{array}{cc} Q^{n-1}_{11}(z,f) & Q^{n-1}_{12}(z,f) \\ Q^{n-1}_{21} (z,f) & Q^{n-1}_{22} (z,f) \end{array}\right).
\end{split}
\end{equation}
Now we consider the neighborhood of a point $x$ with $f(x) = 0$ such that 
\[1 + f^2(x)a(x) \ne 0.
\]
In that neighborhood the four equations given by the entries of \eqref{factorizelocal2} are such that the first three equations imply the fourth by the determinant $1$ condition on both sides. This is the reason why in the matrix we wrote the unspecified entry as $\ast$.

\begin{flalign*}
1+f^2(x) a(x) = Q^{n-1}_{11} (z,f), \tag{1} \\
f^2(x) b(x) - z_{2n}(1+ f^2(x) a(x)) = Q^{n-1}_{12}(z,f), \tag{2} \\
f^2(x) c(x)- z_1 (1+ f^2(x) a(x)) =  Q^{n-1}_{21}(z,f), \tag{3} \\
\ast = Q^{n-1}_{22}(z,f). \tag{4} \\
\end{flalign*}

Since we have $1 + f^2(x)a(x) \neq 0$, equations (2) and (3) can be used to
express $z_1$ and $z_{2n}$ from the other variables $z_2, z_3, \ldots, z_{2n-1}$. This reduces our problem to one equation, namely equation (1).

We continue to analyze this equation. From Lemma \ref{modulof3}, we see that $Q^{n-1}_{11}(z,f) - 1$ is divisible by $f^2$. This is where we use the divisibility condition $f^3 \mid (G - \operatorname{Id})$. If this condition is not fulfilled, the equation has no holomorphic solution, i.e., the reachable points of the fibration $G^{\ast}(\xi)$ would be empty over points $x$ where $f(x) = 0$.

Define $\widetilde{Q}^{n-1}_{11} := (Q_{11}^{n-1}(z,f)-1)/f^2(x)$. Subtracting $1$ from both sides of equation (1) and dividing both of them by $f^2$, we have proved Lemma \ref{one_equation} below. Here dividing by $f^2$ is the second time when we use the fact that we consider the reachable points. Clearly in points $x$ where $f(x) = 0$, the left and right hand side of equation (1) will coincide, being equal to $1$ absolutely independent which values the $z_i$ take. The fibre of $G^{\ast}(\xi)$ is $\C^{2n}$ in these points; however, the reachable points are described by equation $\widetilde{Q}^{n-1}_{11}(z,f) = \text{const}$ and have the same dimension as neighboring fibres. Indeed, in an upcoming step, we will prove that the reachable points (minus the singularity set) form a submersion.
    
\begin{Lem} \label{one_equation} 

If $x$ is such that $f(x) = 0$, then in a neighborhood of $x$ the fibration $G^{\ast}(\xi)$ is isomorphic to the fibration $A : \C^{2n-2} \times X \to \C \times X$ given by
\[(z_2, z_3, \ldots, z_{2n-1},x) \mapsto (\widetilde{Q}^{n-1}_{11} (z,f),x).
\]

\end{Lem}

The claim about the singularity set follows now from the following lemma.
    
\begin{Lem}

The fibres of the function $\widetilde{Q}^{n-1}_{11}(z,f)$ are smooth exactly outside the zero point, i.e, the singularity set of $G^{\ast}(\xi)$ consists of those points where $z_2 = z_3 = \cdots = z_{2n-1} = 0$.

\end{Lem}

\begin{proof}

Outside $\{f = 0\}$ the fibration is just a rescaling of the fibration considered in \cite{IK2}. We therefore omit the proof which can alternatively be achieved by an easy induction argument over the number $n$. Surjectivity of the fibration for $n \ge 4$ for example follows from the simple fact, that $\widetilde{Q}^{n-1}$ is for fixed $x$ a nonconstant polynomial, so it cannot omit any value.

On $\{f = 0\}$ we use the description of $\widetilde{Q}^{n-1}$ given by Lemma \ref{modulof3}. For $f(x) = 0$ we have $\widetilde{Q}^{n-1}_{11} = \sum_{1\le i\le j\le n-1} z_{2i} z_{2j+1}$. If we take the partial derivatives by the even variables, we find
\[\frac{\partial \widetilde{Q}^{n-1}_{11}}{\partial z_{2k}} = \sum_{l=k}^{n-1} z_{2l+1}.
\]
If these partial derivatives are all zero, then using an induction argument starting with $n-1$ down to $2$, shows that all odd variables have to be zero.

In the same way we conclude that if the partial derivatives by odd variables are all zero, then all even variables have to be zero.

\end{proof}

As the last step in the proof of Proposition \ref{ellipticsubmersion}, we show that the restriction of our fibration to $\{f = 0\}$ after removing the singularity set is elliptic. As explained in \cite{Ivarsson:2012} and \cite{IK2}, for any holomorphic function $P$, the vector fields $\Theta_{i,j}:= \dfrac{\partial P}{\partial z_i} \dfrac{\partial}{\partial z_j} - \dfrac{\partial P}{\partial z_j} \dfrac{\partial}{\partial z_i}$ ranging over all pairs of different variables $(i,j)$, span the tangent space at any smooth point of any fibre.

The function $P = \widetilde{Q}^{n-1}_{11}$ is linear in each variable separately, thus these vector fields for $\widetilde{Q}^{n-1}_{11}$ are globally integrable. Hence they are giving the standard Gromov spray.
This finishes the proof of Proposition \ref{ellipticsubmersion}.

\end{proof}

\subsection{The Topological Solution}
\label{ssection:topologicalULfactorization}

In this subsection we prove the following theorem.

\begin{Thm} \label{topologicalfactorization-into-replicas}

Let $U^+, U^-$ and $f \in \O(X)$ be as in the beginning of Section \ref{section:factorizationofsuitableautomorphisms} and consider a continuous bundle automorphism $G \in \SAut_{\operatorname{top}}(E)$ with the properties that $f^4 \mid G - \operatorname{Id}$ and $G$ is strongly null-homotopic. Then
\[G = U^-(h_1) \cdot U^+(h_2) \cdot \ldots \cdot U^-(h_{2k-1}) \cdot U^+(h_{2k})
\]
for some integer $k$ and continuous functions $h_i \in \O(X)$, $i = 1, 2, \ldots, k$.
 
\end{Thm}  

We will adapt the strategy invented by Vaserstein \cite{Vaserstein} using the results of Calder and Siegel \cite{Calder/Siegel:1978}, \cite{Calder/Siegel:1980} on uniform homotopies. Since these results are for maps into $\operatorname{SL}_2(\C)$ and $\operatorname{SU}_2$, and not for maps into bundle automorphisms, we try to turn from the bundle situation to the trivial bundle over $X \setminus \{f = 0\}$ in such a way that our solutions constructed on $X \setminus \{f = 0\}$ extend continuously over $\{f = 0\}$ simply by being the zero solution, i.e., $h_i = 0$ when $f(x) = 0$.
For that we will first solve the problem locally around $\{f = 0\}$ by patching together local solutions of four factors. As said above, we want that in a neighborhood of $\{f = 0\}$ the solution has the property that $h_i (x) = 0$ when $f(x) = 0$ in order to move from the bundle situation to the trivial bundle over $X \setminus \{f = 0\}$. This naively seems trivial; however, the point $z_2 = z_3 = \cdots = z_{2n-1}$ is a singular point of the fibration $G^{\ast}\xi$ introduced  before Proposition
\ref{ellipticsubmersion}. This makes that property difficult to handle. In particular, it cannot be achieved when looking for holomorphic solutions (see Remark \ref{rem:holomorphicinterpolation}). We have to develop a special version of the Whitehead lemma working only for continuous functions, not over any commutative ring. This condition is important to us for the following reason. By the result of Calder and Siegel, any null-homotopy $\alpha_t$ of a map $\alpha : X \to \operatorname{SU}_2$ can be replaced by a uniform null-homotopy $\tilde\alpha_t$. Taking a closer look at the proof of Calder and Siegel, one can achieve that in points $x$ where the homotopy $\alpha_t(x)$ is constantly the identity ${\alpha}_t(x) = \operatorname{Id}$, the uniform homotopy has this property again, namely $\tilde{\alpha}_t (x) = \operatorname{Id}$.

We start with a lemma containing the above described novelties, e.g., our new Whitehead lemma.

\begin{Lem} \label{lem:novelties}
    
\begin{enumerate}
\item For 
\[A(x) = \left(\begin{array}{cc} a(x) & b(x) \\ c(x) & d(x) \end{array}\right) \in \operatorname{SL}_2 (\mathcal{C}^{\mathbb{C}}(X)),
\]
i.e., over the ring of complex valued continuous functions $\mathcal{C}^{\mathbb{C}}(X)$, and under the special condition $a(x) \neq 0$, one can solve the equation 
\[\left(\begin{array}{cc} a(x) & b(x) \\ c(x) & d(x) \end{array}\right) = \left(\begin{array}{cc} 1 & 0 \\ z_1(x) & 1 \end{array}\right)\cdot \ldots \cdot\left(\begin{array}{cc} 1 & z_4(x) \\ 0 & 1 \end{array}\right)\] 
with the interpolation condition
\[A(x) = \operatorname{Id} \implies z_1 (x) = \cdots = z_4 (x) = 0
\]
in the following way:
\begin{itemize}
    
\item $z_1(x) = \dfrac{c(x)-\dfrac{a(x)-1}{\sqrt{|a(x)-1|}}}{a(x)}$;

\item $z_2(x) = \sqrt{|a(x)-1|}$;

\item $z_3(x) = \dfrac{a(x)-1}{\sqrt{|a(x)-1|}}$;

\item $z_4(x) = \dfrac{b(x)-{\sqrt{|a(x)-1|}}}{a(x)}$.

\end{itemize}

\item In case of divisibility $f^3 \mid (A-\operatorname{Id})$, one can solve the equation
\[\left(\begin{array}{cc} 1+f^3 a & f^3 b \\ f^3 c & 1+f^3 d \end{array}\right) = \left(\begin{array}{cc} 1 & 0 \\ f z_1 & 1 \end{array}\right)\cdot \ldots \cdot \left(\begin{array}{cc} 1 & f z_4 \\ 0 & 1 \end{array}\right)\]
under the assumption that $1+f^3 a \ne 0$ with the interpolation condition 
\[A(x) = \operatorname{Id} \implies z_1 (x) = \cdots = z_4 (x) = 0
\]
by choosing
\begin{itemize}
    
\item $z_1 = \dfrac{f^2 c-\dfrac{af}{\sqrt{|af|}}}{1+f^3 a}$;

\item $z_2 = \sqrt{|af|}$;

\item $z_3 = \dfrac{af}{\sqrt{|af|}}$;

\item $z_4 = \dfrac{f^2 b-\sqrt{|af|}}{1+f^3 a}$.

\end{itemize}

\item $\operatorname{Special}$ $\operatorname{Whitehead}$ $\operatorname{Lemma}$. Writing a determinant $1$ invertible diagonal matrix $D$ as a product of four elementary matrices over the ring of complex valued continuous matrices with the interpolation condition 
\[D(x) = \operatorname{Id} \implies z_1 (x) = \cdots = z_4 (x) = 0
\]
is achieved as such. For $\lambda \in \mathcal{C}^{\mathbb{C}}(X)^{\ast}$, the following holds:

\begin{equation*}
\begin{split} D(x) & = \left(\begin{array}{cc} \lambda & \phantom{0} \\ \phantom{0} & \lambda^{-1} \end{array}\right) \\ & = \left(\begin{array}{cc} 1 & 0 \\ \frac{1-\lambda}{\sqrt{|\lambda-1|}} & 1 \end{array}\right) \cdot \left(\begin{array}{cc} 1 & \sqrt{|\lambda-1|} \\ 0 & 1 \end{array}\right) \cdot \left(\begin{array}{cc} 1 & 0 \\ \frac{\lambda-1}{\sqrt{|\lambda-1|}} & 1 \end{array}\right) \cdot \left(\begin{array}{cc} 1 & -\frac{\sqrt{|\lambda-1|}}{\lambda} \\ 0 & 1 \end{array}\right).
\end{split}
\end{equation*}

\end{enumerate}

\end{Lem}

\begin{Rem}\label{rem:holomorphicinterpolation}

 The standard solution for the Whitehead lemma is the following:
\[D = \left(\begin{array}{cc} \lambda & \phantom{0} \\ \phantom{0} & \lambda^{-1} \end{array}\right) = \left(\begin{array}{cc} 1 & 0 \\ -\lambda^{-1} & 1 \end{array}\right) \cdot \left(\begin{array}{cc} 1 & \lambda-1 \\ 0 & 1 \end{array}\right) \cdot \left(\begin{array}{cc} 1 & 0 \\ 1 & 1 \end{array}\right) \cdot \left(\begin{array}{cc} 1 & \lambda^{-1}-1 \\ 0 & 1 \end{array}\right).
\]

This solution does work over any ring. However, the interpolation condition
is not satisfied. Moreover,  there cannot be a Whitehead Lemma for the ring of holomorphic functions on a Stein space $X$ satisfying the interpolation condition. Indeed, suppose $f \in \O(X)$ is an irreducible function on a Stein space $X$. Let $A(x)= \left(\begin{array}{cc} e^{f(x)} & \phantom{0} \\ \phantom{0} & e^{-f(x)} \end{array}\right) \in \operatorname{SL}_2(\O(X))$. This matrix $A$  is equal to the identity on $\{f = 0\}$. However the (1,1)-entry of its difference to the identity matrix is merely divisible by $f$ not by $f^2$. Similar as in Lemma \ref{modulof3}, one can prove that a product of any number of upper/lower replica $U^\pm (z_i (x))$ of the standard unipotent pair with the interpolation property $z_i = 0$ on $\{f = 0\}$ 
($f \mid z_i$) has the property that its entries at $(1,1)$ and $(2,2)$ are divisible by $f^2$.
Thus there cannot be a Special Whitehead Lemma as in $(3)$  for the ring $ \O(X)$.

\end{Rem}

\begin{proof}[Proof of Theorem \ref{topologicalfactorization-into-replicas}]

The first step is to multiply $G$ by continuous replica of $U^\pm (z)$, $z \in \mathcal{C}^{\mathbb{C}}(X)$, so that the resulting automorphism is identity in an open neighborhood of $\{f = 0\}$. 
For that we patch together local solutions using cutoff functions. We can assume that there is a finite covering $X = \cup_{i=1}^M U_i$ so that the bundle $E$ is trivial (even holomorphically) over $U_i$. 

Recall the notations from Section \ref{section:pairs}. In some trivialization $E \simeq U_{i} \times \mathbb{C}^2$, the sections $s_1, s_2$ of $E$ are represented by vectors $s_1^{i}(x), s_2^{i}(x) \in \C^2$ which may become linearly dependent in some points $x$ contained in the set $\{ f = 0\}$. We denote by $S_i(x)$ the matrix which has the above sections on its columns. We have that 
\[N^+(x) = S_i(x) \cdot\left(\begin{array}{cc} 0 & f(x) \\ 0 & 0 \end{array}\right)\cdot S_i^{-1}(x) = S_i(x) \cdot\left(\begin{array}{cc} 0 & f_i(x) \\ 0 & 0 \end{array}\right)\cdot S_i^{\#}(x)
\]
and
\[N^-(x) = S_i(x) \cdot\left(\begin{array}{cc} 0 & 0 \\ f(x) & 0 \end{array}\right)\cdot S_i^{-1}(x) = S_i(x) \cdot\left(\begin{array}{cc} 0 & 0 \\ f_i(x) & 0 \end{array}\right)\cdot S_i^{\#}(x),
\]
where $S_i^{\#}(x)$ is the adjoint of $S_i(x)$ and $f_i = \dfrac{f}{\det S_i} \in \O (U_i)$ (see Section \ref{section:pairs}). Remember that $U^{\pm} = \operatorname{Id} + N^{\pm}$ and $U^{\pm}(h) = \operatorname{Id} + h\cdot N^{\pm}$.

As in the beginning of the proof of Proposition \ref{ellipticsubmersion}, we make the following considerations. On $U_i$ outside the analytic subset, where the sections $s_1, s_2$ are linearly dependent equation \ref{factorize} is equivalent to the equation
\[S_i^{-1}(x)\operatorname{Mat}(G) (x)S_i(x) =  S_{i}^{-1}(x)U^-(z_1(x))S_i(x) \cdot S_{i}^{-1}(x)U^+ (z_2(x))S_i(x) \cdot \ldots .
\]
where $\operatorname{Mat}(G)(x) \in SL_2(\C)$ is the matrix representing the automorphism $G(x)$ in local coordinates. Since $G - \operatorname{Id}$ is divisible by $f^4$, we will write 
\[\operatorname{Mat}(G)(x) = \operatorname{Id} + f^4(x) \left(\begin{array}{cc} a_1(x) & b_1(x) \\ c_1(x) & d_1(x) \end{array}\right).
\]
Hence
\[\operatorname{Id} + f^4(x) \cdot S_i^{-1} (x) \left(\begin{array}{cc} a_1(x) & b_1(x) \\ c_1(x) & d_1(x) \end{array}\right) S_i (x) = \left(\begin{array}{cc} 1 & 0 \\ z_1f(x) & 1 \end{array}\right) \cdot \left(\begin{array}{cc} 1 & z_2f(x) \\ 0 & 1 \end{array}\right) \cdot \ldots,
\]
or equivalently
\begin{equation}
\label{factorizelocalcont}  
\begin{split}
\operatorname{Id} + f^3(x) \cdot f_i (x) S_i^{\#}(x) & \left(\begin{array}{cc} a_1(x) & b_1(x) \\ c_1(x) & d_1(x) \end{array}\right) S_i (x) = \\ & = \left(\begin{array}{cc} 1 & 0 \\ z_1f(x) & 1 \end{array}\right) \cdot \left(\begin{array}{cc} 1 & z_2f(x) \\ 0 & 1 \end{array}\right) \cdot \ldots . 
\end{split}
\end{equation}

Since both sides of this last equation are continuous on all of $U_i$, the identity principle for continuous  maps ($X \setminus \{ f=0\}$ is dense in $X$) shows that our problem (\ref{factorize}) is equivalent to this last equation (\ref{factorizelocalcont}) on all of $U_i$ (not only where the the sections are linearly independent). 


We introduce the notation
\[\left(\begin{array}{cc} a(x) & b(x) \\ c(x) & d (x) \end{array}\right) :=  S_i^{\#} (x) \left(\begin{array}{cc} a_1(x) & b_1(x) \\ c_1(x) & d_1(x) \end{array}\right) S_i (x)
\]
Now the reachable points of $G^{\ast}(\xi)$ are described by 
\begin{equation}
\label{factorizelocalcont1}    
\operatorname{Id} + f^3(x) f_i (x) \left(\begin{array}{cc} a(x) & b(x) \\ c(x) & d (x) \end{array}\right) = \left(\begin{array}{cc} 1 & 0 \\ z_1f(x) & 1 \end{array}\right) \cdot \left(\begin{array}{cc} 1 & z_2f(x) \\ 0 & 1 \end{array}\right) \cdot \ldots. 
\end{equation}

We now cover $\{f = 0\} $ by open sets $\Omega_i \Subset \widetilde{\Omega} _i$, $i = 1, 2, \ldots, M$, each relatively compact in the other, such that every $\widetilde{\Omega}_i$ is contained in $U_i$ and so that on $\widetilde \Omega_i$, the corresponding function $1 + f^3 (x) f_i(x) a(x) \ne 0$. In principle we should have denoted $a, b, c, d$ with an index $i$, but it should be clear in which trivialization we are working, so we omit the $i$ for better readability. 

We start on $\Omega_1$ and solve, using Lemma \ref{lem:novelties} (2), the equation
\[\left(\begin{array}{cc} 1+f^3 f_i a & f^3 f_i b \\ f^3 f_i c & 1+f^3 f_i d \end{array}\right) = \left(\begin{array}{cc} 1 & 0 \\ f z_1 & 1 \end{array}\right)\cdot \ldots \cdot\left(\begin{array}{cc} 1 & f z_4 \\ 0 & 1 \end{array}\right).\]
We choose

\begin{itemize}    
\item $z_1 = \dfrac{f^2 f_ic-\dfrac{a f_i f}{\sqrt{|a f_i f|}}}{1+f^3 f_i a}$;

\item $z_2 = \sqrt{|a f_i f|}$;

\item $z_3 = \dfrac{a f_i f}{\sqrt{|a f_i f|}}$;

\item $z_4 = \dfrac{f^2 f_i b-\sqrt{|af_i f|}}{1+f^3 f_i a}$.
\end{itemize}

In order to globalize, we select a continuous cutoff function $\chi : X \to [0,1]$ such that $\chi \equiv 1$ on $\Omega_1$ and $\chi \equiv 0$ on $X \setminus \widetilde{\Omega}_1$ and consider the global continuous bundle automorphisms

\begin{equation}
U_j := U_j^\pm ( \chi (x) z_j (x)) \label{def:unipotents}, \ j = 1, 2, 3, 4.
\end{equation}
Then on $\Omega_1$ we have
\begin{equation}
\label{a} G = U_1 \cdot U_2 \cdot U_3 \cdot U_4.
\end{equation} 
Let us mention the following property:
\begin{equation}
\label{keep_identity} G(x) = \operatorname{Id} \implies \text{the whole  homotopy}\, t \mapsto (U_j)_t := U_j^\pm (t \chi(x)z_j (x)) \equiv \operatorname{Id}.
\end{equation} 
This follows from the important interpolation property that if $G(x) = \operatorname{Id}$, then $z_i(x) = 0$, $\forall i$.

We will check in Lemma \ref{divisible} below that the product of the $U_j$ still satisfies the divisibility condition $f^4 \mid (U_1 \cdot U_2 \cdot U_3 \cdot U_4 - \operatorname{Id})$ globally on $X$. On $\Omega_1$ this follows from equation \ref{a}. Then replacing $G$ by $G_1 := G U_4^{-1} \cdot U_3^{-1} \cdot U_2^{-1} \cdot U_1^{-1}$ we have achieved that this new automorphism $G_1$ satisfies the same conditions as $G$ and, in addition, $G_1$ is identically $\operatorname{Id}$ on $\Omega_1$.

Now we repeat  the same process we did on $U_1$ on $U_2$, working with $G_1$ instead of $G$. We find four global uniform replica $U_i$ of $U^{\pm}$ such that on $\Omega_2$ we have
\begin{equation}\label{a1}
G_1 = U_1 \cdot U_2 \cdot U_3 \cdot U_4.
\end{equation} 

Since each $U_j$ and the whole homotopy $(U_j)_t$ is the identity when $G_1 (x) = \operatorname{Id}$, we find that $G_2 := G_1 U_4^{-1} \cdot U_3^{-1} \cdot U_2^{-1} \cdot U_1^{-1}$ is identical to $\operatorname{Id}$ not only on $\Omega_2$ (as equation \ref{a1} implies), but also on $\Omega_1 \cup \Omega_2$. As by Lemma \ref{divisible} below, the product of the $U_j$'s still satisfies the divisibility condition $f^4 \mid (U_1 \cdot U_2 \cdot U_3 \cdot U_4 - \operatorname{Id})$ globally on $X$, we have the same for $G_2$.

After $M$ steps we have replaced the original $G$ by a new null-homotopic automorphism, for simplicity call it again $G$, such that in addition to the divisibility of the whole null-homotopy $f^4 \mid G_t-\operatorname{Id}$, we have that on a whole open neighborhood $\Omega$ of $\{f = 0\}$, the automorphism $G$ is equal to $\operatorname{Id}$. 

By this last property, using continuity, we can find a neighborhood $\Omega_1$ with $\{f = 0\} \subset \Omega_1 \subset \Omega$ such that $G(\Omega_1) \subset U(\operatorname{Id})$, where $U(\operatorname{Id})$ is an open neighborhood of the identity section in $\Aut(E)$ which is fibrewise contractible to the identity section. For example one can take a neighborhood $U(\operatorname{Id})$ which is mapped by $\log$ biholomorphically to a convex neighborhood of the zero section in the corresponding Lie algebra bundle $\End^0(E)$ and take the contraction 
\[\alpha_t (x,A) := (x, \exp((1-t) \cdot \log(x, A)),\ x \in \Omega_1,\ A \in \SAut(E_x),\ t \in [0,1].
\]
Now we change the null-homotopy of $G$ to a homotopy which is identically $\operatorname{Id}$ in a whole neighborhood $\Omega_2$ of $\{f = 0\}$. For that we choose an open neighborhood $\Omega_2$ such that $\{f = 0\} \subset \Omega_2 \subset \overline{\Omega}_2 \subset \widetilde{\Omega}_2 \subset \Omega_1 \subset \Omega$ and a continuous function $\chi : X \to [0, 1]$ with $\chi\vert_{\Omega_2} \equiv 0$ and $\chi\vert_{X \setminus \widetilde{\Omega}_2} \equiv 1$. We define the new homotopy 
$G_t^{\text{new}}(x) : = \alpha_{1-\chi (x)}$. Since the function $\chi$ only differs from zero inside the set $\Omega_1$ where $G$ is the identity, we have not changed $G = G_0$. We have however changed the homotopy $G_t$ to be constant identity in the neighborhood $\Omega_2$ of $\{f = 0\}$, $\forall t \in [0,1]$. This allows us to turn from the bundle situation to the trivial bundle over $X_f := X \setminus \{f = 0\}$. Using the sections $s_1, s_2$ as coordinates we present $G_t$ by a map $\widehat{G}_t : X_f \to \operatorname{SL}_2(\C)$. Now we can follow the strategy of Vaserstein from \cite{Vaserstein} with the necessary adjustments. We need to keep all unipotent replicas constant identity on $\Omega_2 \setminus \{f = 0\} \subset X_f$ in order to be able to extend the solution continuously from $X_f$ to $X$. This is done by simply setting all functions $h$ occurring in the replicas $U^\pm (h)$ to be $0$ on $\{f = 0\}$.

Firstly, by multiplying with replicas, we want to change the map $\widehat{G}_t$ to a map into the maximal compact subgroup $\operatorname{SU}_2$. For that we use the Gram-Schmidt process to multiply by a matrix of the form 
\[\left(\begin{array}{cc} 1 & 0 \\ f (x) z_1(x,t) & 1 \end{array}\right),
\]
where $z_1(x, t) : X_f \times [0,1] \to \C$ is a continuous function in order to make the columns of $\widehat{G}_t$ orthogonal. The factor $f(x)$ is nonzero on $X_f$ and thus the function from Gram-Schmidt can just be divided by $f$ in order to get $z_1$. Observe that when $\widehat{G}_t$ is the identity, Gram-Schmidt gives $z_1 = 0$. Thus $z_1$ can be extended to all of $X$ by setting it $0$ on $\{f = 0\}$. Hence the matrix above gives rise to a continuous replica $U^-(z_1)$. Replacing $\widehat{G}_t$ by $U^-(z_1) \cdot \widehat{G}_t$, we can assume that the columns of the homotopy $\widehat{G}_t$ are orthogonal. 

Next we want to use the Whitehead Lemma to normalize the columns to length $1$. Here we use our Special Whitehead Lemma \ref{lem:novelties} $(3)$ which gives a product of four unipotent matrices with the property that when $\widehat{G}_t$ is the identity, then the functions $z_1 (x,t), \ldots, z_4 (x,t)$ are zero. As above we can divide the functions by $f$, extend them to all of $X$ by zero and thus obtain well defined continuous replica $U^\pm$ such that when we multiply $\widehat{G}_t$ by them we have made the columns of the homotopy of length $1$. So we can assume that the homotopy is unitary $\widehat{G}_t : X_f \to \operatorname{SU}_2$ and the homotopy is still identically $\operatorname{Id}$ on $\Omega_2 \setminus \{f = 0\}$.

Now we use the result of Calder and Siegel \cite{Calder/Siegel:1978}, \cite{Calder/Siegel:1980} to replace the homotopy $\widehat{G}_t$ by a uniform homotopy. Their result is for more general target spaces $Y$ instead of $\operatorname{SU}_2$. Their result requires that $Y$ needs to be compact and have finite fundamental group $\pi_1 (Y)$. Since $\operatorname{SU}_2$ is diffeomorphic to $S^3$, it is compact and simply connected. A bundle version of this result was proved by Hultgren and Wold in \cite{WoHu}.

We need an extra property of the uniform homotopy constructed by Calder and Siegel: \textit{When the homotopy $\widehat{G}_t$ in a point $x$ is constant the identity, then the uniform homotopy has the same property}. 

We believe this is true in the general setting. In our case $Y = \operatorname{SU}_2$, one can adapt the proof of Calder and Siegel in the following way. The main ingredient in their proof is the result that the pointed loop space $C_0 (S^1, \operatorname{SU}_2)$ is homotopy equivalent to a CW-complex $K$ of finite type (has finitely many cells in each dimension, the complex itself is infinite). For the case of spheres, which is our case since $\operatorname{SU}_2$ is diffeomorphic to $S^3$, this result about $C_0 (S^1, S^n)$ is very well explained in the book of Milnor \cite{Milnor:1963}. The complex $K$ has exactly one cell in each dimension $0, n-1, 2\,(n-1), 3\,(n-1), \ldots$ and $K = C_0 \cup C_{n-1} \cup C_{2(n-1)} \cup \ldots$. First one retracts the pointed loop space $C_0 (S^1, S^n)$ to a certain subset consisting of piecewise geodesic loops. Then in that subset, one finds the cells $C_i$. The cell of dimension zero is the constant identity loop. The cell $C_{n-1}$ consists of geodesic loops which pass exactly once through the north pole of the sphere (when the marked point is the south pole), the cell $C_{2(n-1)}$ consists of geodesic loops which pass exactly twice through the north pole of the sphere and so on. The null-homotopy $\widehat{G}_t$ is a continuous map $X_f \to C_0(S^1, S^3)$, it can be homotoped to a map to $K$ and, since our Stein space $X_f$ is a finite dimensional CW-complex, it can be further homotoped to a map into the $l$-skeleton of $K$, as soon as $l$ is greater or equal to the real dimension of $X_f$. Since the $l$-skeleton is compact, we have a uniform null-homotopy. All these homotopies one can perform in such a way that the zero cell $C_0$ of $K$ is not moved, in other words, where the homotopy was identically south pole it remains identically south pole.

Now by the uniformity of the homotopy $\widehat{G}_t : X_f \times [0,1] \to \operatorname{SU}_2$, given any open neighborhood $V$ of $\operatorname{Id}$, we find real numbers $0 = t_0 < t_1 < t_2 < \ldots < t_{n-1} < t_n = 1$ such that $\widehat{G}_{t_{i+1}} \cdot \widehat{G}_{t_i}^{-1} (X_f) \subset V$, $\forall i = 0, 1, \ldots, n-1$. For $V$ we choose the open subset of $\operatorname{SU}_2$ defined by 
\[\left\{\left(\begin{array}{cc} a & b \\c & d \end{array}\right) \in \operatorname{SU}_2  \, : \, a \ne 0\right\}.
\]
Then since $\widehat{G}_n = G$ and $\widehat{G}_0 = \operatorname{Id}$,

\begin{equation}\label{nearidentity}
\widehat{G} = (\widehat{G}_{n} \widehat{G}_{n-1}^{-1}) (\widehat{G}_{n-1} \widehat{G}_{n-2}^{-1}) \cdots (\widehat{G}_{i} \widehat{G}_{i-1}^{-1}) \cdots (\widehat{G}_1 \widehat{G}_0^{-1})
\end{equation}
and each expression $\widehat{G}_{i} \widehat{G}_{i-1}^{-1}$ in brackets can be written by Lemma \ref{lem:novelties} $(1)$ as a product of four unipotent matrices which are identically $\operatorname{Id}$ on $\Omega_2 \cap X_f$. These unipotent matrices give rise to unipotent replica $U^\pm (z)$, where the functions $z \equiv 0$ on $\Omega_2 \cap X_f$ and therefore can be continuously extended to all of $X$. Finally, this leads by equation \ref{nearidentity} to
\[G = \prod_{i=1}^n U_{i,1}^- U_{i,2}^+ U_{i,3}^- U_{i,4}^+.
\]
This finishes the proof of Theorem \ref{topologicalfactorization-into-replicas}.

\end{proof}

\begin{Lem} \label{divisible}

The unipotent replicas $U_1, \ldots, U_4$ defined in \ref{def:unipotents} satisfy $f^4 \mid (U_1 \cdot U_2 \cdot U_3 \cdot U_4 -\operatorname{Id})$. 
\end{Lem}

\begin{proof}
    
We have to prove that if the $z_i$'s are defined as 
\begin{itemize}
    
\item $z_1 = \dfrac{f^2 f_ic-\dfrac{a f_i f}{\sqrt{|a f_i f|}}}{1+f^3 f_i a}$,

\item $z_2 = \sqrt{|a f_i f|}$,

\item $z_3 = \dfrac{a f_i f}{\sqrt{|a f_i f|}}$,

\item $z_4 = \dfrac{f^2 f_i b-\sqrt{|af_i f|}}{1+f^3 f_i a}$,

\end{itemize}
in order to satisfy 
\begin{equation}
 \label{eq:product}   
\left(\begin{array}{cc} 1+f^3 f_i a & f^3 f_i b \\ f^3 f_i c & 1+f^3 f_i d \end{array}\right) = \left(\begin{array}{cc} 1 & 0 \\ f z_1 & 1 \end{array}\right) \cdot \ldots \cdot \left(\begin{array}{cc} 1 & f z_4 \\ 0 & 1 \end{array}\right) 
\end{equation}
then the difference of 
\[\left(\begin{array}{cc} 1 & 0 \\ \chi f z_1 & 1 \end{array}\right) \cdot \ldots \cdot \left(\begin{array}{cc} 1 &\chi f z_4 \\ 0 & 1 \end{array}\right)
\]
to the identity is divisible by $f^3 f_i$. Indeed, going back as from equation \ref{factorizelocalcont1} to equation \ref{factorizelocalcont} from the divisibility by $f_i$, we conclude divisibility by an extra factor $f$, so we get divisibility by $f^4$. We first calculate
\[\left(\begin{array}{cc} 1 & 0 \\ f z_1 & 1 \end{array}\right)\cdot \ldots \cdot \left(\begin{array}{cc} 1 & f z_4 \\ 0 & 1 \end{array}\right) =
\left(\begin{array}{cc} 1+f^2 z_2 z_3  & f z_2  + f z_4  + f^3 z_2 z_3 z_4  \\  f z_1  +f z_3  +  f^3 z_1 z_2 z_3  & \ast \end{array}\right).
\]
We omit to write the entry in position $(2,2)$ since the determinant $1$ condition will imply the divisibility for that entry given the right divisibility for the other three entries (see Lemma \ref{lem:last_entry} below). Next we calculate 
\[\begin{split}
& \left(\begin{array}{cc} 1 & 0 \\ \chi f z_1 & 1 \end{array}\right)\cdot \ldots \cdot \left(\begin{array}{cc} 1 &\chi f z_4 \\ 0 & 1 \end{array}\right) = \\ & =
\left(\begin{array}{cc} 1+f^2 z_2 z_3 \chi^2 & f z_2 \chi + f z_4 \chi + f^3 z_2 z_3 z_4 \chi^3 \\  f z_1 \chi +f z_3 \chi +  f^3 z_1 z_2 z_3 \chi^3 & \ast \end{array}\right).
\end{split}
\]
By the formulas for $z_2$ and $z_3$ above, we see that the product $z_2 z_3$ is divisible by $f_i f$. Thus the $(1,1)$ entry, namely $1+f^2 z_2 z_3$  has the right form. By equation (\ref{eq:product}), we find that
\[f^3 f_i b =  f z_2  + f z_4  + f^3 z_2 z_3 z_4.
\]
Since $f f_i \mid z_2 z_3$, we conclude that $f^2 f_i \mid z_2 + z_4$. This implies that the entry $(1,2)$, namely $ f z_2 \chi + f z_4 \chi + f^3 z_2 z_3 z_4 \chi^3$, is divisible by $f^3 f_i$. The $(2,1)$ entry is dealt with analogously. As remarked before, the right form of the not calculated $(2,2)$ entry follows from Lemma \ref{lem:last_entry}.

\end{proof}

\begin{Lem} \label{lem:last_entry}

If for a determinant $1$ matrix of continuous functions of the form  \[\left(\begin{array}{cc} 1+a_{11} & a_{12} \\ a_{21} & 1+a_{22} \end{array}\right)
\]
and a continuous function $g$ holds $g  \mid a_{11}$, $g \mid a_{12}$,
$g \mid a_{21}$, then $g \mid a_{22}$.

\begin{proof}

 The determinant $1$ condition is 
 \[(1+ a_{11}) \cdot (1+a_{22}) - a_{12} a_{21} = 1.
 \]
 Writing $a_{11}= g \Tilde{a}_{11}$, $a_{12}=  g \Tilde{a}_{12}$, $a_{21} = g \Tilde{a}_{21}$ for continuous functions $\Tilde{a}_{ij}$, we find after some algebraic manipulations 
\[a_{22} = g(g \Tilde{a}_{12} \Tilde{a}_{21} - \Tilde{a}_{11} a_{22} -\Tilde{a}_{11}).\]
 This shows that $a_{22}$ is divisible by $g$.
 
\end{proof}

\end{Lem}

\subsection{Application of the Oka Principle for Stratified Elliptic Submersions}
\label{ssection: holomorphicULfactorization}

\begin{proof}[Proof of Theorem \ref{factorization-into-replicas}]
 
 We have to find some $n\in \N$ and a holomorphic section $S_{\text{hol}}$ in the fibration $G^\ast (\xi) : G^\ast (\C^{2n}) \to X $. By Proposition \ref{ellipticsubmersion} the reachable points of the fibration $(G^{\ast}(\mathbb{C}^{2n}) \setminus \operatorname{Sing}, G^{\ast}\psi_{2n}, X)$
form a stratified elliptic submersion. By Theorem \ref{topologicalfactorization-into-replicas}, we find a $k\in \N$ and a continuous section $S_{\text{cont}}$ of the fibration $G^\ast (\xi) : G^\ast (\C^{2k}) \to X $. Of course this section will meet the singularity set, since due to our method of proof we took care of the fact that on $\{f = 0\}$, it goes through the singularity set. We add to the $2k$ continuous replica whose product gives $G$, four more replica, namely $U^-(1), U^+(0), U^-(-1), U^+(0)$ whose product is identity $U^-(1) \cdot U^+(0) \cdot  U^-(-1) \cdot U^+(0) = \operatorname{Id}$ and thus we get a continuous section $S_{\text{cont}}$ of $G^\ast (\xi) : G^\ast (\C^{2k+4}) \to X $. This section avoids the singularity set over any point of $X$, since the singularity set is the set where $z_2 = z_3 = \ldots z_{2k+3} = 0$ and we have $z_{2k+3} \equiv -1$. Now setting $n = k+2$, we can use Proposition \ref{ellipticsubmersion} together with the Oka Principle for Stratified Elliptic Submersions, namely Theorem \ref{t:forstneric} to find a holomorphic section $S_{\text{hol}}$ in the fibration $G^\ast (\xi) : G^\ast (\C^{2n}) \to X$ homotopic to $S_{\text{cont}}$. This finishes the proof of Theorem \ref{factorization-into-replicas}.

\end{proof}

\begin{proof}[Proof of Theorem \ref{Thm:main}]

 The necessity of being a null-homotopy is easily seen. Indeed, for a unipotent bundle automorphism $U$ we have $U = \operatorname{Id} + N$, where $N$ is nilpotent. The homotopy $t \mapsto \operatorname{Id} + t N$ connects $U$ to the identity. Thus any product of unipotent bundle automorphisms is null-homotopic.
 
 The sufficiency part follows from Proposition \ref{existence of pairs}, Theorem \ref{product of suitable} and Theorem \ref{factorization-into-replicas}.

\end{proof}

\begin{Thm}\label{thm:uniform bound}

There exists a natural number $K$ (depending on the dimension $n$ of the Stein space $K = K(n)$) such that for any Stein space $X$ with $\dim(X) = n$ and any rank $2$ holomorphic vector bundle $E \to X$ over $X$, any null-homotopic holomorphic vector bundle automorphism $F \in \SAut(E)$ is 
a  product of unipotent holomorphic automorphisms $u_i \in \U(E)$, $i = 1, 2, \ldots, K$,
\[F(x) = u_1 (x) \cdot u_2 (x) \cdot \ldots \cdot u_K (x).
\]

\begin{proof}
    


By contradiction, let us assume that there are infinitely many Stein spaces $X_i$ of dimension $n$, together with holomorphic vector bundles $E_i \to X_i$ and null-homotopic holomorphic vector bundle automorphism $F_i \in \SAut(E_i)$ which are products of not less than $i$ unipotent automorphisms. Then taking the disjoint union, namely $X = \dot\cup X_i$ (which is of dimension $n$) and $E =\dot\cup E_i$, all $F_i$'s define together a unipotent bundle automorphism $F$ of $E$. By Theorem \ref{Thm:main}, the automorphism $F$ would factorize into finitely many, say $K$, unipotent automorphism, which by restricting to $X_i$ would mean the same for each $F_i$. This contradicts our assumption on the $F_i$.

\end{proof}

\end{Thm}

\begin{Rem} 

To establish the optimal number $K(n)$ is very difficult. For Stein spaces of dimension $1$, all vector bundles are trivial and thus the optimal number is the same as for maps $X \to \operatorname{SL}_2(\C)$. This number has been found to be $4$ in \cite{IK2}. Thus we see that $K(1) = 4$.

For Stein spaces of dimension $2$, by our method we would find a trivialization of the bundle consisting of $3$ subsets together with suitable pairs $U^-, U^+$. Then we would factorize the automorphisms in a product of three automorphism $G_i$ each suitable for a pair. Presumably the methods in \cite{IK2} would yield a factorization into five unipotent factors giving an upper bound $K(2) \le 15$. For that, one would need to analyze the fibration from \ref{diagram} for an odd number of factors, namely $5$ factors. In the present article, we have chosen to analyze only the case of even number of factors since our main purpose was to prove the existence of a factorization, not to find the optimal number of factors.
    
\end{Rem}


\begin{thebibliography}{99}

  
\bibitem[B68]{Bass} Bass, H.,
\emph{Algebraic K-theory.} W. A. Benjamin, Inc., New York-Amsterdam 1968 xx+762 pp. 
     
\bibitem[BMS67]{BMS} Bass, H.; Milnor, J.; Serre, J.-P.,
\emph{Solution of the congruence subgroup problem for $SL_n(n \ge 3)$ and $Sp_{2n}(n\ge 2)$.} Inst. Hautes \'Etudes Sci. Publ. Math. No. 33 (1967), 59–-137. 

\bibitem[CS78]{Calder/Siegel:1978}
Allan Calder and Jerrold Siegel, \emph{Homotopy and uniform homotopy}, Trans. Amer. Math. Soc. \textbf{235} (1978), 245--270

\bibitem[CS80]{Calder/Siegel:1980}
Allan Calder and Jerrold Siegel, \emph{Homotopy and uniform homotopy. II.}, Proc. Amer. Math. Soc. \textbf{78}, (1980), no.2, 288--290

\bibitem[Ca58]{Ca} Cartan, Henri, 
\emph{Espaces fibr\'{e}s analytiques},
{Symposium internacional de topolog\'{\i}a algebraica {I}nternational symposium on algebraic topology}, Universidad Nacional Aut\'{o}noma de M\'{e}xico and UNESCO, Mexico City, (1958), 97--121.

 \bibitem[Co66]{Cohn} Cohn, P. M.,
\emph{On the structure of the $GL_2$ of a ring.} Inst. Hautes \'Etudes Sci. Publ. Math. No. 30 (1966), 5–-53.
 
    

\bibitem[DK19]{DK} Doubtsov, Evgueni; Kutzschebauch, Frank, 
\emph{Factorization by elementary matrices, null-homotopy and products of exponentials for invertible matrices over rings.} Anal. Math. Phys. 9 (2019), no. 3, 1005–-1018.

\bibitem[FR66]{Okasche Paare} O. Forster, K. J. Ramspott,
\emph{Okasche Paare von Garben nicht-abelscher Gruppen.} Invent. Math. 1, 260-286 (1966).

\bibitem[For10]{Forstneric:2010} Franc Forstneri{\v{c}},
\emph{The {O}ka principle for sections of stratified fiber bundles.}, Pure Appl. Math. Q. \textbf{6} (2010), 843--874.

\bibitem[For11]{Forstneric:2011} Franc Forstneri{\v{c}}, 
\emph{Stein manifolds and holomorphic mappings.} Springer-Verlag, 2011.

\bibitem[FP01]{Forstneric:2001} Franc Forstneri{\v{c}} and Jasna Prezelj,
\emph{Extending holomorphic sections from complex subvarieties.} Math. Z. \textbf{236} (2001), 43--68.

\bibitem[FP02]{Forstneric:2002} Franc Forstneri{\v{c}} and Jasna Prezelj, \emph{Oka's principle for holomorphic submersions with sprays.} Math. Ann. \textbf{322} (2002), 633--666.

\bibitem[Gro89]{Gromov:1989} Mikhael Gromov, 
\emph{Oka's principle for holomorphic sections of elliptic bundles.} J. Amer. Math. Soc. \textbf{2} (1989), 851--897.

\bibitem[GMV91]{Grunewald:1991} Fritz Grunewald, Jens Mennicke, and Leonid  Vaserstein,
\emph{On symplectic groups over polynomial rings.} Math. Z. \textbf{206} (1991), 35--56.

\bibitem[HW20]{WoHu} Hultgren, Jakob; Wold Erlend,
\emph{Unipotent Factorization of Vector Bundle Automorphisms.}, to appear in Int. J. of Math.

\bibitem[IK12b]{IK2}  Ivarsson, Björn; Kutzschebauch, Frank,
\emph{On the number of factors in the unipotent factorization of holomorphic mappings into $\mbox{SL}_2(\mathbb{C})$.} Proc. Amer. Math. Soc. 140 (2012), no. 3, 823--838.

\bibitem[IK12a]{Ivarsson:2012} Bj{\"o}rn Ivarsson and Frank Kutzschebauch, \emph{Holomorphic factorization of mappings into $\mbox{SL}_n(\mathbb{C})$.} Ann. of Math. (2) \textbf{175} (2012), 45--69.
  
\bibitem[IKL19]{Ivarsson:2019} Bj{\"o}rn Ivarsson, Frank Kutzschebauch and Erik L{\o}w, 
\emph{Factorization of symplectic matrices into elementary factors.}  
Proc. Amer. Math. Soc. 148 (2020), no. 5, 1963–-1970.

\bibitem[JMW86]{Jones:86} P.W.Jones, D.Marshall, and T.Wolff, 
\emph{Stable rank of the disc algebra}, Proc. Amer. Math. Soc. \textbf{96}, No.4 (1986), 603--604.

\bibitem[KR88]{Klein:1988} Manfred Klein and Karl Josef Ramspott, 
\emph{Ein Transformationssatz f{\"u}r Idealbasen holomorpher Funktionen.}
Bayer. Akad. Wiss. Math.-Natur. Kl. Sitzungsber. \textbf{1987} (1988), 93--100.

\bibitem[Kop78]{Kopeiko} Kopeĭko, V. I., 
\emph{Stabilization of symplectic groups over a ring of polynomials.} (Russian) Mat. Sb. (N.S.) 106(148) (1978), no. 1, 94–-107, 144.
 


 
\bibitem[KS19]{KuSt}Kutzschebauch, Frank; Studer, Luca,
\emph{Exponential factorizations of holomorphic maps.} Bull. Lond. Math. Soc. 51 (2019), no. 6, 995--1004. 
 
\bibitem[Ku14]{Kutzschebauch} Kutzschebauch, Frank, 
\emph{Flexibility properties in complex analysis and affine algebraic geometry.} Automorphisms in birational and affine geometry, 387--405, Springer Proc. Math. Stat., 79, Springer, Cham, 2014. 

\bibitem[Lei20]{Leiterer} Leiterer, J\"urgen,  
\emph{On holomorophic matrices on bordered Riemann surfaces.},  Bull. Lond. Math. Soc. 53 (2021), no. 3, 906–-916
 
\bibitem[Mil71]{Milnor:1971}  Milnor, J.,
\emph{Introduction to algebraic K-theory.}, Princeton University Press,
N.J., University of Tokyo Press, Tokyo (1971), Annals of Mathematics Studies No.72.

\bibitem[Mil63]{Milnor:1963} J. Milnor,
\emph{Morse theory.} 
Based on lecture notes by M. Spivak and R. Wells Annals of Mathematics Studies, No. 51 Princeton University Press, Princeton, N.J. (1963) vi+153 pp. 

\bibitem[MoRu18]{Mort}
Mortini, Raymond; Rupp, Rudolf,
\emph{Logarithms and exponentials in the matrix algebra ${\mathcal M}_2(A)$.}  
Comput. Methods Funct. Theory 18 (2018), no. 1, 53–87.

\bibitem[Stu20a]{Studer2}  Studer, Luca, \emph{ A homotopy theorem for Oka theory.} Math. Ann. 378 (2020), no. 3-4, 1533–-1553. 

\bibitem[Stu20b]{Studer1}  Studer, Luca, \emph{ A splitting lemma for coherent sheaves.} Anal. PDE 14 (2021), no. 6, 1761–-1772. 

\bibitem[Sus77]{Suslin} A.\ A.\ Suslin, 
\emph{The structure of the special linear group over rings of polynomials.} {\it Izv.\ Akad.\ Nauk SSSR Ser.\ Mat.\ }{\textbf 41} (1977), no.\ 2, 235--252, 477, (English translation, {\it Math.\ USSR Izv.\ }{\textbf 11} (1977), 221--238.) 


\bibitem[ThurVas88]{Thurston} 
Thurston, W.  Vaserstein, L. 
\emph{On K1-theory of the Euclidean space.} 
Topology Appl. 23 (1986), no. 2, 145–148. 

\bibitem[Vas88]{Vaserstein} L.\ Vaserstein, 
\emph{ Reduction of a matrix depending on parameters to a diagonal form by addition operations.}  Proc.\ Amer.\ Math.\ Soc.\ {\bf 103} (1988), no.\ 3, 741--746.

\bibitem[VS13]{VS} Vavilov, N. A.; Stepanov, A. V.,
\emph{Linear groups over general rings. I. Generalities.} (Russian) Zap. Nauchn. Sem. S.-Peterburg. Otdel. Mat. Inst. Steklov. (POMI) 394 (2011), Voprosy Teorii Predstavleniĭ Algebr i Grupp. 22, 33–-139, 295; (English translation in J. Math. Sci. (N.Y.) 188 (2013), no. 5, 490–-550). 

\end{thebibliography}
\end{document}